\documentclass[10pt]{article}
\usepackage[margin=1in]{geometry} 
\usepackage{a4wide}
\usepackage[latin1]{inputenc} 
\usepackage{algorithmic} 
\usepackage{amsmath}
\usepackage{amsfonts} 
\usepackage{amssymb} 
\usepackage{graphicx}
\usepackage{amsthm} 
\usepackage{mathabx}
\usepackage{mathtools} 
\usepackage[ruled]{algorithm2e}
\usepackage{enumerate} 
\usepackage{color} 
\usepackage{subfigure}
\usepackage{epigraph} 
\usepackage{hyperref}
\usepackage{authblk}
\usepackage{multirow}
     
\mathtoolsset{showonlyrefs=true}

\newtheorem{definition}{Definition}[section]

\newtheorem{remark}{Remark}[section]
\newtheorem{lemma}{Lemma}[section]
\newtheorem{theorem}{Theorem}[section]
\newtheorem{fact}{Fact}[section]

\newtheorem{corollary}{Corollary}[section]
\newtheorem{alg}{Algorithm}[section]

\newgeometry{left=2cm, right=2cm, top=2cm, bottom=2cm}

\providecommand{\keywords}[1]{\textbf{\textit{Keywords: }} #1}
\providecommand{\msc}[1]{\textbf{\textit{Subject Classification: }} #1}

\title{Preconditioners for 
  Saddle Point Problems on Truncated Domains in Phase Separation Modelling}

\author[1]{Pawan Kumar\thanks{pawan.kumar@iiit.ac.in}}
\affil[1]{Old Mumbai Highway}
\affil[1]{International Institute of Information Technology, Hyderabad}
\affil[1]{Hyderabad 500032, India}
 
\date{\today}

\begin{document}

\maketitle

\begin{abstract}
  The discretization of Cahn-Hilliard equation with obstacle potential leads to
  a block $2 \times 2$ {\em non-linear} system, where the $(1,1)$ block has a
  non-linear and non-smooth term. Recently a globally convergent Newton Schur
  method was proposed for the non-linear Schur complement corresponding to this
  non-linear system. The solver may be seen as an inexact Uzawa method which has
  the falvour of an active set method in the sense that the active sets are
  first identified by solving a quadratic obstacle problem corresponding to the
  $(1,1)$ block of the block $2 \times 2$ nonlinear system, and a new decent
  direction is obtained after discarding the active set region. The problem
  becomes linear on nonactive set, and corresponds to solving a linear saddle
  point problem on truncated domains. For solving the quadratic obstacle
  problem, various optimal multigrid like methods have been proposed. In this
  paper solvers for the truncated saddle point problem is considered. Three
  preconditioners are considered, two of them have block diagonal structure, and
  the third one has block tridiagonal structure. One of the block diagonal
  preconditioners is obtained by adding certain scaling of stiffness and mass
  matrices, whereas, the remaining two involves Schur complement.  Eigenvalue
  bound and condition number estimates are derived for the preconditioned
  untruncated problem. It is shown that the extreme eigenvalues of
  the preconditioned truncated system remain bounded by the extreme eigenvalues
  of the preconditioned untruncated system. Numerical experiments confirm the
  optimality of the solvers.
\end{abstract}
\\
\msc{65F08, 35P15, 35J86}

\section{Introduction} 
The Cahn-Hilliard equation was first proposed in 1958 by Cahn and Hilliard
\cite{Cahn1958} to study the phase separation process in a binary alloy. Here
the term phase stands for the concentration of different components in an
alloy. It has been empirically observed that the concentration changes from a
given mixed state to a spatially separated two phase state when the alloy under
preparation is subjected to a rapid cooling below a critical temperature. This
rapid reduction in the temperature the so-called deep quench limit has been
found to be modeled efficiently by obstacle potential proposed by Oono and Puri
\cite[Fig. 7, p. 439]{Oono1987} in 1987, and analyzed by Blowey and Elliot
\cite[p. 237, (1.14)]{Blowey1991}. The phase separation has been noted to be
highly non-linear, and the obstacle potential
emulates the nonlinearity and non-smoothness that is empirically observed and
much desired in numerical smulations.  Consequently, handling the non-smoothness
as well as designing robust iterative procedure have been a subject of much
active research during last decades. Assuming semi-implicit time discretizations
\cite{Blowey1992} to alleviate the time step restrictions, most of the proposed
methods essentially differ in the way the nonlinearity and non-smoothness are
handled. There seems to be two main approaches to handle the non-smoothness:
regularization around the non-smooth region and subsequently using a variant of
smooth solvers, for example, as in \cite{Bosch2014}, or an active set like
approach \cite{Graeser2009}, i.e., where one identifies the active sets via a
nonlinear solver, subsequently, after discarding the active set nodes, we obtain a
reduced (or truncated) problem which is linear. Moreover, the global convergence
of the nonlinear solver may be ensured by a proper damping parameter, for example, as done in \cite{Graeser2009}.

The non-linear problem corresponding to Cahn-Hilliard equation with obstacle
potential could be written as a non-linear system in block $2 \times 2$ matrix
form as follows:
\begin{align}
  \begin{pmatrix}
    F & B^T \\
    B & -C
  \end{pmatrix}
  \begin{pmatrix}
    u \\ w
  \end{pmatrix}
  \ni
  \begin{pmatrix}
    f \\ g
  \end{pmatrix}, \quad u, w \in \mathbb{R}^n, \label{eqn:saddle}
\end{align}
where $u$ and $w$ are unknowns corresponding to order parameter and chemical
potential respectively, $F = A + \partial I_K,$ where $I_K$ denotes the
indicator functional for $u$ corresponding to the admissible set $K.$ We note that $F(\cdot)$ is a set valued mapping due to
the presence of set-valued operator $\partial I_K,$ hence, we have inclusion  in \eqref{eqn:saddle} instead of equality. The matrix $A$ corresponds to
Laplacian with Neumann boundary conditions perturbed by a rank one
term, and is multiplied by a parameter corresponding to interface
width. On the other hand, $C$ is also Laplacian with Neumann boundary condition,
but multiplied by the time step parameter. Both nonlinearity and non-smoothness
are due to $\partial I_K$ in $F.$ Various non-linear and nonsmooth solvers have
been proposed for \eqref{eqn:saddle} \cite{Banas2009, Bosch2014}.  

By nonlinear Gaussian elimination of $u,$ the system above could
be reduced to a nonlinear Schur complement system in $w$ variables
\cite{Graeser2009}, where the ``negative" nonlinear Schur complement is given by
$C + B(F)^{-1}B^T.$ Here $(\cdot)^{-1}$ is understood as inversion in the
nonlinear sense. In \cite{Graeser2009}, a globally convergent Newton method is
proposed for this nonlinear Schur complement system, which is interpreted as a
preconditioned Uzawa iteration.  To solve the inclusion $F(x)
\ni y$ corresponding to the quadratic obstacle problem, many methods have been
proposed such as block Gauss-Seidel \cite{Barrett2004, Glowinski2008}, monotone
multigrid method \cite{Kornhuber1994, Kornhuber1996, Mandel1984}, truncated
monotone multigrid \cite{Graeser2009b}, and truncated Newton multigrid
\cite{Graeser2009b}. 

Once active sets are identified, the corresponding rows and columns are anhilated,
we then obtain a reduced linear
system as follows
\begin{align}
  \begin{pmatrix}
    \hat{A} & \hat{B}^T \\
    \hat{B} & -C
  \end{pmatrix}
  \begin{pmatrix}
    \hat{u} \\ \hat{w}
  \end{pmatrix}
  =
  \begin{pmatrix}
    \hat{f} \\ \hat{g}
  \end{pmatrix}, \quad \hat{u}, \hat{w} \in \mathbb{R}^n. \label{eqn:red}
\end{align}
Here a solution to \eqref{eqn:red} is a new descent direction in the Uzawa iteration.  By a choice of appropriate step size along this descent direction, global convergence of the Uzawa method is ensured. As the active sets change during each iteration, the linear system, and hence the
preconditioners need to be updated. 

In this paper, our
goal is to design effective preconditioner and hence an iterative solver for
\eqref{eqn:red} such that the convergence rate is independent of problem
parameters and mesh size. There are several classes of preconditioners: multigrid \cite{kumar2014,KUMAR20132251}, domain decomposition \cite{das2019,kumar2011a}, deflation based preconditioners \cite{das2021,katyan2020,niu2010,kumar2014z}. Three preconditioners are considered; two of them
involves Schur complement.  Two of these preconditioners have block diagonal structure and they correspond to 
non-standard norms proposed in \cite{Zulehner2011}. To approximate the Schur
complement, we consider an approximation proposed in \cite{Bosch2014}.  It turns out that
the building blocks of these preconditioners are same, their analysis is
remarkably similar, even though, they may look structurally different from the outset. Eigenvalue
bound and condition number estimates are derived for these preconditioners for
the untruncated problem. The obtained eigenvalue bounds seem to be tight when
compared to numerically computed extreme eigenvalues. Subsequently, it is shown
that the extreme eigenvalues of the preconditioned truncated problem are bounded
from above and below by the extreme eigenvalues of the corresponding
preconditioned untruncated problem. We also verify the effectiveness of these
preconditioners numerically for various evolutions.

The rest of this paper is organized as follows. In Section \ref{sec:ch}, we
describe the Cahn-Hilliard model with obstacle potential, we discuss the time
and space discretizations and variational formulations. In Section
\ref{sec:solver}, we discuss briefly the solver for Cahn-Hilliard with obstacle
problem. The preconditioners for the truncated linear saddle point problem \eqref{eqn:red}, and their eigenvalue analysis are discussed in Section \ref{sec:precon}. Finally, in Section \ref{sec:numexp}, we show
numerical experiments with the proposed preconditioners.
 
\section{\label{sec:not}Notations}
Let SPD and SPSD denote symmetric positive definite and symmetric positive semi
definite respectively. Let $\kappa(M)$ denote the condition number of SPD matrix $M.$ For $x
\in \mathbb{R}, |x|$ denotes the absolute value of $x,$ whereas, for any set
$\mathcal{K},$ $|\mathcal{K}|$ denotes the number of elements in $\mathcal{K}.$
Let $Id \in \mathbb{R}^{n \times n}$ denote the identity matrix. Let
$\mathbf{1}$ denote $[1,1,1,\dots,1].$ For a matrix $Z \in
\mathbb{R}^{n \times n}$ with all real eigenvalues, the eigenvalues will be denoted and ordered as follows
\begin{align}
  \lambda_{1}(Z) \leq \lambda_{2}(Z) \leq \dots \leq \lambda_{n}(Z).
\end{align}

\section{\label{sec:ch} Cahn-Hilliard Problem with Obstacle Potential}
\subsection{The Model}
We will consider a model for phase separation of two components in a binary alloy mixture. Here phase stands for concentration of two components in the mixture. Let $u_1, u_2 \in [0,1]$ be the concentration of two components in
the mixture, then we set $u=u_1 - u_2 \in [-1,1].$ The phase separation is modelled using Cahn-Hilliard equations, which is obtained by $H^{-1}$ gradient flow of 
Ginzburg-Landau (GL) energy functional which is given as follows
\begin{align}
  E(u) = \int_{\Omega} \frac{\epsilon}{2}|\nabla u|^2 +
  \psi(u) \, dx, \quad \Omega=(0,1) \times (0,1). \label{eqn:gle}
\end{align}
 Here the constant
$\epsilon$ relates to interfacial thickness, and the obstacle potential $\psi,$
 which is used to model deep quench phenomena is given as follows
\begin{align}
  \psi(u) = \psi_0(u) + I_{[-1, 1]}(u), \quad \text{where} \quad \psi_0(u) =
  \frac{1}{2}(1-u^2).
\end{align}
Here the subscript $[-1,1]$ of indicator function $I$ above denotes the range of
admissible values of $u.$  Here
$I_{[-1,1]}(u)$ is defined as follows
\begin{align}
  I_{[-1,1]} = \begin{cases} 0, \: \text{if} \: u(i) \in [-1,1] \\
    \infty, \: \text{otherwise}.
  \end{cases}
\end{align}
Moreover, $u_1+u_2$ is assumed to be conserved.  The
$H^{-1}$ gradient flow of $E$ leads to the Cahn-Hilliard equation in PDE form
\begin{align}
  \partial_t u &= \Delta w, \\
  w &= - \epsilon \Delta u + \psi^{'}_{0}(u) + \mu, \\
  \mu &\in \partial I_{[-1,1]}(u), \\
  \frac{\partial u}{\partial n} &= \frac{\partial w}{\partial n} = 0 \:
  \text{on} \: \partial \Omega.
\end{align}
The unknowns $u$ and $w$ are called order parameter and chemical potential
respectively.  For a given $\epsilon > 0,$ final time $T>0,$ and initial
condition $u_0 \in \mathcal{K},$ where
\begin{align} \mathcal{K}=\{ v \in H^1(\Omega) \, : \, |v| \leq
  1\}, \label{eqn:setK} \end{align} the equivalent initial value problem for
Cahn-Hilliard equation with obstacle potential interpreted as variational
inequality reads
\begin{align}
  \left \langle \frac{du}{dt}, v \right \rangle + (\nabla w, \nabla v)
  &= 0, \: \forall v \in H^1(\Omega), \label{eqn:ch31} \\
  \epsilon (\nabla u, \nabla(v-u)) - (u, v-u) &\geq (w,v-u), \: \forall v \in
  \mathcal{K}, \label{eqn:ch32}
\end{align}
where we use the notation $\langle \cdot, \cdot \rangle$ to denote the duality
pairing of $H^1(\Omega)$ and $H^1(\Omega)^{'}.$ Note that we used the fact that
$\psi^{'}_{0}(u) = -u$ in the second term on the left of inequality
\eqref{eqn:ch32} above. The inequalities \eqref{eqn:ch31} and \eqref{eqn:ch32}
are defined on constrained set $\mathcal{K},$ the variational inequality of
first kind is also equivalently represented on unconstrained set using the indicator
functional \cite[p. 2]{Glowinski2008}. The existence and uniqueness of the
solution of \eqref{eqn:ch31} and \eqref{eqn:ch32} above have been established in
Blowey and Elliot \cite{Blowey1991}. We next consider an appropriate
discretization in time and space for \eqref{eqn:ch31} and \eqref{eqn:ch32}.
\subsection{Time and space discretizations}\label{timeAndSpace}
We consider a fixed non-adaptive grid in time interval $(0,T)$ and in space
$\Omega$ defined in \eqref{eqn:gle}. The time step $\tau = T/N$ is kept
uniform. We consider the semi-implicit Euler discretization in time and finite
element discretization in space as in Barrett et. al. \cite{Barrett2004} with
triangulation $\mathcal{T}_h$ with the following spaces 
\begin{align}
  \mathcal{S}_h &= \left\{ v \in C(\bar{\Omega}) \, : v|_T ~\mbox{is
      linear} \quad \forall T \in T_h \right\}, \\
  \mathcal{P}_h &= \left\{ v \in L^2(\Omega) \, : \, v|_T \,
    \mbox{is constant} \quad \forall T \in \mathcal{T} \in \mathcal{T}_h \right\}, \\
  \mathcal{K}_h &= \left\{ v \in \mathcal{P}_h \, : \, \left|\, v|_T \, \right|
    \leq 1 \quad \forall T \in \mathcal{T}_h \right\} = \mathcal{K} \cap
  \mathcal{S}_h \subset \mathcal{K},
\end{align}
which leads to the following discrete Cahn-Hilliard problem with obstacle
potential: \\\\
Find $u^k_h \in \mathcal{K}_h, w^k_h \in \mathcal{S}_h$ s.t.
\begin{align}
  \langle u^k_h, v_h \rangle + \tau (\nabla w^k_h, \nabla v_h)
  &= \langle u^{k-1}_h, v_h \rangle, \: \forall v_h \in \mathcal{S}_h, \label{eqn:dche1} \\
  \epsilon (\nabla u^k_h, \nabla (v_h - u^k_h)) - \langle w^k_h, v_h - u^k_h
  \rangle &\geq \langle u^{k-1}_h, v_h - u^k_h \rangle, \: \forall v_h \in
  \mathcal{K}_h \label{eqn:dche2}
\end{align}
holds for each $k=1,\dots,N.$ The initial solution $u^0_h \in \mathcal{K}_h$ is
taken to be the discrete $L^2$ projection $\langle u^0_h, v_h \rangle = (u_0,
v_h), \forall v_h \in \mathcal{S}_h.$

Existence and uniqueness of the discrete Cahn-Hilliard equations has been
established in \cite{Blowey1992}. The discrete Cahn-Hilliard equation is
equivalent to the set valued saddle point block $2 \times 2$ nonlinear system
\eqref{eqn:saddle} with $F = A + \partial I_{\mathcal{K}_h}$ and
\begin{align}
  A = \epsilon (\langle \lambda_p, 1 \rangle \langle \lambda_p, 1 \rangle
  + (\nabla \lambda_p, \nabla \lambda_q))_{p,q \in \mathcal{N}_h}, \\
  B = (\langle \lambda_p, \lambda_q \rangle)_{p,q \in \mathcal{N}_h}, \: C =
  \tau ((\nabla \lambda_p, \nabla \lambda_q))_{p,q \in
    \mathcal{N}_h}. \label{eqn:stiff}
\end{align}
We write the above in more compact notations as follows
\begin{align}
  A = \epsilon(K+mm^T), \quad B = M, \quad C = \tau K, \label{eqn:short_not}
\end{align}
where $m = \langle \lambda_p, 1 \rangle,$ $M$ and $K$ are usual notations for
mass and stiffness matrices respectively.
\section{\label{sec:solver}Iterative solver for Cahn-Hilliard with obstacle
  potential}
In \cite{Graeser2009}, a nonsmooth Newton Schur method is proposed which is also
interpreted as a preconditioned Uzawa iteration. For a given time step $k,$ the
Uzawa iteration reads:
\begin{align}
  u^{i,k} &= F^{-1} (f^k - B^T w^{i,k}), \label{eqn:uzawa1} \\
  w^{i+1,k} &= w^{i,k} + \rho^{i,k} \hat{S}^{-1}_{i,k}(Bu^{i,k} - Cw^{i,k} -
  g^k) \label{eqn:uzawa2}
\end{align}
for the saddle point problem \eqref{eqn:saddle}. Here $i$ denotes the $i^{th}$
Uzawa step, and $k$ denotes the $k^{th}$ time step. Here $f^k$ and $g^k$ are
defined as follows
\begin{align}
  \langle f, v_h\rangle = \langle u_h^{k-1}, v_h \rangle, \quad \langle g, v_h
  \rangle = - \langle u_h^{k-1}, v_h \rangle.
\end{align}
The time loop starts with an initial value for $u^{0,0}$ which can be taken
arbitrary as the method is globally convergent, and with the initial value
$w^{0,0}$ obtained from \eqref{eqn:uzawa2}. The Uzawa iteration requires three main computations that we describe
below.
\subsection{Computing $u^{i,k}$}
The first step \eqref{eqn:uzawa1} corresponds to solving a quadratic obstacle
problem interpreted as a minimization problem as follows
\begin{align}
  u^{i,k} = \mbox{arg}~\underset{v \in K}{\mbox{min}} \left(\frac{1}{2}\langle
    Av, v \rangle - \langle f^k - B^T w^{i,k}, v \rangle \right).
\end{align}
As mentioned in the introduction, this problem has been extensively studied
during last decades \cite{Barrett2004, Graeser2009b, Kornhuber1994,
  Kornhuber1996}.
\subsubsection{Algebraic Monotone Multigrid for Obstacle Problem}
To solve the quadratic obstacle problem \eqref{eqn:uzawa1}, we use the monotone
multigrid method proposed in \cite{Kornhuber1994}. In Algorithm \ref{alg:mmg}, we
describe an algebraic variant of the method. The algorithm performs one V-cycle
of multigrid; it takes $u^i$ from the previous iteration, and outputs the
improved solution $u^{i+1}.$ The initial set of interpolation operators are
constructed using aggregation based coarsening \cite{Kumar2014}.
\subsection{Computing $\hat{S}^{-1}_{i,k}(Bu^{i,k} - Cw^{i,k} - g^k)$}
The quantity $d^{i,k} = \hat{S}^{-1}_{i,k} (Bu^{i,k} - Cw^{i,k} - g^k)$ in
\eqref{eqn:uzawa2} is obtained as a solution of the following reduced linear
block $2 \times 2$ system:
\begin{align}
  \begin{pmatrix}
    \hat{A} & \hat{B}^T \\
    \hat{B} & -C
  \end{pmatrix}
  \begin{pmatrix}
    \tilde{u}^{i,k} \\ d^{i,k}
  \end{pmatrix} =
  \begin{pmatrix}
    0 \\ g + Cw^{i,k} - Bu^{i,k} \label{eqn:rls}
  \end{pmatrix},
\end{align}
where
\begin{align}
  \hat{A} = TAT + \hat{T}, \quad \hat{B} =
  BT. \label{eqn:trunc_mass_stiff}
\end{align}
Here truncation matrices $T$ and $\hat{T}$ are defined as follows:
\begin{align} \label{eqn:TandThat} T = \mbox{diag}
  \begin{pmatrix*}[l]
    0, \quad \mbox{if}~u^{i,k}(j) \in \{-1, 1\} \\
    1, \quad \mbox{otherwise}
  \end{pmatrix*}, \quad \hat{T} = \mbox{diag}
  \begin{pmatrix*}[l]
    1, \quad \mbox{if}~T_{jj}=0 \\
    0, \quad \mbox{otherwise}
  \end{pmatrix*}, \quad j=1,\dots,|N_h|,
\end{align}
where $u^{i,k}(j)$ is the $j$th component of $u^{i,k},$ and $T_{jj}$ is the
$j$th diagonal entry of $T.$ In words, $\hat{A}$ is the matrix obtained from
$A$ by replacing the $i$th row and $i$th column by the unit vector $e_i$
corresponding to the active sets identified by diagonal entries of $T.$
Similarly, $\hat{B}$ is the matrix obtained from $B$ by annihilating rows,
and $\hat{B}^T$ is the matrix obtained from $B$ by annihilating columns.
Rewriting untruncated version of \eqref{eqn:rls} in simpler notation as follows
\begin{align}
  \left( \begin{array}{cc}
      \epsilon \bar{K} & M                              \\
      M & -\tau K
    \end{array} \right) \left( \begin{array}{c}
      x                              \\ y
    \end{array} \right) = \left( \begin{array}{c}
      0 \\ b
    \end{array} 
  \right),
\end{align}
where $\bar{K} = K + mm^T.$ By a change of variable $y' = y/\epsilon,$ we obtain
\begin{align}
  \left( \begin{array}{cc}
      \bar{K} & M                                       \\
      M & -\eta K
    \end{array}\right) \left( \begin{array}{c}
      x                              \\ y'
    \end{array} \right) = \left( \begin{array}{c}
      0 \\ b
    \end{array} 
  \right),
\end{align}
where $\eta = \epsilon \cdot \tau.$
Furthermore, we modify the $(2,2)$ term of the system matrix above as follows
\begin{align}
  - \eta K = - \eta K - \eta mm^T + \eta mm^T = - \eta \bar{K} + (\eta^{1/2}m)
  (\eta^{1/2}m^T) = - \eta \bar{K} + \tilde{m} \tilde{m}^T,
\end{align}
where $\tilde{m} = \eta^{1/2}m.$
Now the untruncated system may be rewritten as
\begin{align}
  \bar{\mathcal{A}} = \begin{pmatrix}
    \bar{K} & M \\
    M & -\eta \bar{K}
  \end{pmatrix} + \bar{m} \bar{m}^T =: \mathcal{A} +
  \bar{m}\bar{m}^T, \label{eqn:simplified}
\end{align}
where $\bar{m}^T = [0, \tilde{m}^T] \in \mathbb{R}^{2|\mathcal{N}_h|}$ is a rank one term with proper extension by
zero. Now we are in a position to use Sherman-Woodbury inversion for matrix plus
rank-one term to invert $\bar{\mathcal{A}}$.  In this paper, we shall develop efficient solvers to solve the
truncated system
\begin{align}
  \hat{\mathcal{A}}v = z, \label{eqn:trunc}
\end{align}
where $\hat{\mathcal{A}}$ is defined next.  We denote
\begin{align}
  \label{eqn:not}
  \hat{K} = TKT + \hat{T}, \quad A = \bar{K}, \quad B = M, \quad
  \hat{A} = TAT + \hat{T}, \quad \hat{B} = TB.
\end{align}
Note that the notation $A$ appearing above has now been redefined.
Thus, the truncated system corresponding to \eqref{eqn:simplified} reads
\begin{align}
  \hat{\mathcal{A}} + \bar{m}\bar{m}^T = \begin{pmatrix}
    \hat{A} & \hat{B} \\
    \hat{B}^T & -\eta A
  \end{pmatrix} + \bar{m} \bar{m}^T.  \label{eqn:trun_simplified}
\end{align}
Thus, Sherman-Woodbury inversion formula may be used to invert $\hat{\mathcal{A}} +
\hat{m}\hat{m}^T,$ and it is enough to find an efficient solver for
\eqref{eqn:trunc}.
\subsection{Computing step length $\rho^{i,k}$}
The step length $\rho^{i,k}$ can be computed using a bisection method. We refer the
interested reader to \cite{Graeser2011}[p. 88].
\begin{algorithm}
  \begin{algorithmic}[1]
    \REQUIRE Let $V_1 \subset V_1 \subset V_2 \cdots V_m$ and let $r_m, b_m \in
    V_m,$ \REQUIRE $u^i, i>0$ solution from previous cycle or $u^0$ a given initial
    solution \STATE Compute residual: \: $r_m = b_m - A_m u^i$ \STATE Compute
    defect obstacles: \begin{align} \begin{cases}
        \underline{\delta}_m = \underline{\psi} - u^i \\
        \bar{\delta}_m = \bar{\psi} - u^i
      \end{cases} \end{align} \FOR{$\ell=m,\cdots,2$} \STATE Projected
    Gauss-Seidel Solve using Algorithm \ref{alg:pgs}:
    \begin{align}
      (D_\ell + L_\ell + \partial I_{\mathcal{K}^\ell})v_\ell = r_\ell,
    \end{align}
    where
    \begin{align}
      \mathcal{K}^\ell = \left\{\, v \in \mathbb{R}^{n_\ell} \, \middle| \:
        \underline{\delta}_{\ell} \leq v \leq \bar{\delta}_{\ell} \,
      \right\}.
    \end{align}
    \STATE Update
    \begin{align}
      \begin{cases} r_{\ell} &:= r_{\ell} - A_{\ell} v_{\ell} \\
        \underline{\delta}_{\ell - 1} &:= \underline{\delta}_{\ell} - v_{\ell} \\
        \bar{\delta}_{\ell - 1} &:= \bar{\delta}_{\ell} - v_{\ell}
      \end{cases}
    \end{align}
    \STATE Restrict and compute new obstacle
    \begin{align}
      \begin{cases}
        r_{\ell-1} = P^T_{\ell-1}r_{\ell}\\
        (\underline{\delta}_{\ell-1})_i := \max \left\{
          (\underline{\delta}_{\ell-1})_j \: \middle| \:
          (P_{\ell-1})_{ji} \neq 0 \right\}, \: i=1,\dots,n_{\ell} -1 \\
        (\bar{\delta}_{\ell-1})_i := \min \left\{
          (\bar{\delta}_{\ell-1})_j \: \middle| \: (P_{\ell-1})_{ji} \neq 0
        \right\}, \: i=1,\dots,n_{\ell} -1
      \end{cases}
    \end{align}
    \ENDFOR
    \STATE Solve
    \begin{align}
      (D_1 + L_1 + \partial I_{\mathcal{K}^1})v_1 = r_1
    \end{align}
    \FOR{$\ell = 2, \cdots, m$} \STATE Add corrections
    \begin{align}
      v_{\ell} := v_{\ell} + P_{\ell-1}v_{\ell-1}
    \end{align}
    \ENDFOR
    \STATE Compute
    \begin{align}
      u^{i+1} = u^i + v_m
    \end{align}
    \ENSURE improved solution $u^{i+1}$
  \end{algorithmic}
  \caption{\label{alg:mmg}Monotone Multigrid (MMG) V cycle}
\end{algorithm}
\begin{algorithm}
  \begin{algorithmic}[1]
    \REQUIRE $A \in \mathbb{R}^{n_{\ell} \times n_{\ell}}, \quad b,
    \underline{\psi}, \bar{\psi} \in \mathbb{R}^{n_{\ell}},$ \quad current
    iterate $x^i \in \mathbb{R}^{n_{\ell}}$ \ENSURE new iterate $x^{i+1} \in
    \mathbb{R}^{n_{\ell}}$ \STATE Compute residual: \[r := b - Ax^i \] \STATE
    Compute defect obstacles: \[\underline{\psi} := \underline{\psi} - x^i \] \[
    \bar{\psi} := \bar{\psi} - x^i \] \FOR{$i=1:n_{\ell}$}
    \FOR{$j=1:i$} \STATE Compute $y_i$
    \begin{align}
      y_i = \begin{cases}
        \max \left(\min \left((r_i - A_{ij}y_j)/A_{ii}, \: \bar{\psi}_i \right), \: \underline{\psi}_i \right), \quad \mbox{if} \: A_{ii} \neq 0, \\
        0, \quad \mbox{otherwise}
      \end{cases}
    \end{align}
    \ENDFOR
    \ENDFOR
    \STATE $x^{i+1} = x^i+y$
  \end{algorithmic}
  \caption{\label{alg:pgs} $x^{i+1}$ $\leftarrow$ PGS($x^i, A, \underline{\psi},
    \bar{\psi}, b$)}
\end{algorithm}
\subsection{\label{sec:mixedfem} Mixed Finite Element Formulation of Reduced
  Linear System}
We choose suitable Hilbert spaces for trial
and test spaces as follows
\begin{align}
  \hat{V} &= \{v \in H^1(\Omega) : v|_{\Omega_A}=0 \}, \quad Q =
  H^1_0(\Omega), \quad \Omega_A = \Omega \setminus \Omega_I, \label{eqn:V}
\end{align}
where $\Omega_A = \Omega \setminus \Omega_I$ is the domain where truncation takes place. Indeed, if $\Omega_A$ is empty,
then $\Omega_I = \Omega,$ and we set $\hat{V} = V = H^1(\Omega).$ The weak
form of the partial differential equations corresponding to the
truncated system \eqref{eqn:trunc} reads
\begin{align}
  \mbox{Find}~ (u, \lambda) \in \hat{V} \times H^1(\Omega): \\
  \hat{a}(u, v) + \hat{b}(v, \lambda) &= f(v) \quad \mbox{for all} \: v \in \hat{V}, \label{eqn:weak1}\\
  \hat{b}(u, q) - c(\lambda, q) &= g(q) \quad \mbox{for all} \: q \in
  Q, \label{eqn:weak2}
\end{align}
where
\begin{equation}
  \begin{aligned} \label{eqn:abc}
    \hat{a}(u,v) &= \left( (\nabla u, \nabla v) + \int_{\Omega} u \int_{\Omega} v \right)
    = \left((\nabla u, \nabla v) + \langle u, \mathbf{1} \rangle \langle v, \mathbf{1}\rangle \right) , \\ 
    c(\lambda,q) &= \eta \left( \left(\nabla \lambda, \nabla q \right) + 
      \int_{\Omega} u \int_{\Omega} v \right), \quad \hat{b}(v,\lambda) = (v, \lambda).
  \end{aligned}
\end{equation}
The mixed variational problem above can also be written as a variational form on
product spaces
\begin{align}
  \mbox{Find}~ x \in \hat{V} \times Q: \quad \hat{\mathcal{B}} (x,y) =
  \hat{\mathcal{F}}(y) \quad \forall y \in \hat{V} \times
  H^1(\Omega), \label{eqn:mixed_bilin}
\end{align}
where $\hat{\mathcal{B}}$ and $\hat{\mathcal{F}}$ are defined as follows
\begin{align}
  \hat{\mathcal{B}} (z,y) = \hat{a}(w,v) + \hat{b}(v,r) +
  \hat{b}(w,q) - c(r,q), \quad \hat{\mathcal{F}}(y) = f(v) + g(q)
\end{align}
for $y=(v,q) \in \hat{V} \times Q$ and $z = (w,r) \in \hat{V} \times Q.$
The corresponding bilinear form for the untruncated system is given as follows
\begin{align}
  \mathcal{B} (z,y) = a(w,v) + b(v,r) + b(w,q) - c(r,q), \quad \mathcal{F}(y) =
  f(v) + g(q)
\end{align}
for $y=(v,q) \in V \times Q$ and $z = (w,r) \in V \times Q,$ where $V =
H^1(\Omega).$ 
The mixed variational problem corresponding to untruncated system now reads
\begin{align}
  \mbox{Find}~ x \in V \times Q: \quad \mathcal{B} (x,y) =
  \mathcal{F}(y) \quad \forall y \in V \times
  H^1(\Omega). \label{eqn:mixed_bilin2}
\end{align}
In the rest of this paper, we shall consider norms proposed in
\cite{Zulehner2011} as follows
\begin{align}
  ((v,q), (w,r))_{\hat{X}} = (v,w)_{\hat{V}} +
  (q,r)_Q, \label{eqn:normX}
\end{align}
where $(\cdot, \cdot)_{\hat{V}}$ and $(\cdot, \cdot)_Q$ are inner products
of Hilbert spaces $\hat{V}$ and $Q,$ respectively.  As will see shortly such
norms lead to block diagonal preconditioners.  The boundedness condition that
we seek for the mixed problem for the untruncated problem reads
\begin{align}
  \sup_{0 \neq z \in X} \sup_{0 \neq y \in X} \frac{\mathcal{B}(z,y)}{\|z\|_X
    \|y\|_X} \leq \bar{c}_x < \infty. \label{eqn:bound}
\end{align}
We have the following conjecture for the truncated problem
\begin{align}
  \sup_{0 \neq z \in \hat{X}} \sup_{0 \neq y \in \hat{X}}
  \frac{\hat{\mathcal{B}}(z,y)}{\|z\|_{\hat{X}} \|y\|_{\hat{X}}}
  \leq \sup_{0 \neq z \in X} \sup_{0 \neq y \in X}
  \frac{\mathcal{B}(z,y)}{\|z\|_X \|y\|_X} \leq \bar{c}_x <
  \infty. \label{eqn:DoesThisHold1}
\end{align}
Similarly, for well-posedness of \eqref{eqn:mixed_bilin2}, following well known Babuska-Brezzi condition
needs to be satisfied
\begin{align}
  \inf_{0 \neq z \in X} \sup_{0 \neq y \in X}
  \frac{\mathcal{B}(z,y)}{\|z\|_{X} \|y\|_{X}}
  \geq \underbar{c}_x > 0. \label{eqn:bb}
\end{align}
Similarly, it is not evident whether the following inequality must hold.
\begin{align}
  \underline{c}_x \leq \inf_{0 \neq z \in X} \sup_{0 \neq y \in X}
  \frac{\mathcal{B}(z,y)}{\|z\|_X \|y\|_X} \leq \inf_{0 \neq z \in \hat{X}}
  \sup_{0 \neq y \in \hat{X}}
  \frac{\mathcal{\hat{B}}(z,y)}{\|z\|_{\hat{X}}
    \|y\|_{\hat{X}}}. \label{eqn:DoesThisHold2}
\end{align}
We shall call the norms $\| \cdot \|_X$ and $\| \cdot \|_{\hat{X}}$ optimal
if the constants $\bar{c}_x$ and $\underline{c}_x$ remain independent of the
problem parameters: $\tau$ and $\epsilon$, moreover, in the
discrete space, also remains independent of the mesh size $h.$ The reason why we
are interested in the inequalities \eqref{eqn:DoesThisHold1} and
\eqref{eqn:DoesThisHold2} is that any optimal norm that is found for untruncated
problem shall lead to optimal norm for truncated problem as well. Note that
boundary of untruncated problem has certian regularity (for example Lipschitz
continuity), but for the truncated problem no such regularity is to be assumed, because
the truncations are assumed to be arbitrary. Our plan of attack is to use the
approach of \cite{Zulehner2011}, which is readily applicable for our untruncated
problem.  Although, \eqref{eqn:DoesThisHold1} and \eqref{eqn:DoesThisHold2} are left
as conjecture for the moment, we shall try to answer this in
the discrete case: we shall show a related result that the extreme eigenvalues
of the truncated preconditioned operator are bounded by the extreme
eigenvalues of the corresponding untruncated preconditioned operator. Hence, in
the following, we first derive optimal
preconditioners for the untruncated problem.

We shall provide equivalent conditions as in \cite{Zulehner2011} for
\eqref{eqn:bound} and \eqref{eqn:bb} that lead to deriving the optimal norms,
i.e., optimal preconditioners.  But first we introduce some notations for
operators corresponding to bilinear forms.  It is easy to see that $V \times Q$
is a Hilbert space itself as $V$ and $H^1(\Omega)$ are themselves Hilbert
spaces. It is convenient to associate linear operators for the bilinear forms
$a,b,$ and $c$ as follows
\begin{equation}
\begin{aligned} \label{eqn:ops}
 \langle Aw, v \rangle &= a(w,v), \quad A \in L(V, V^*), \\
 \langle Bw, q \rangle &= b(w,q), \quad B \in L(V, Q^*), \\
 \langle Cr, q \rangle &= c(r,q), \quad C \in L(Q, Q^*), \\
 \langle B^*r, v \rangle &= \langle Bv, r \rangle, \quad B^* \in L(Q, V^*).
\end{aligned}
\end{equation}
Consequently, the operator corresponding to mixed bilinear form $\mathcal{A},$
and the right hand side $\mathcal{F}$ (reusing the notation) in operator notation are given as
follows
\begin{align}
 \mathcal{A} = \begin{pmatrix}
                A & B^* \\
		B & -C  
               \end{pmatrix}, \quad \mathcal{F} = 
               \begin{pmatrix}
                f \\ g
               \end{pmatrix}, \quad x = 
               \begin{pmatrix}
                u \\ p
               \end{pmatrix}.               
\end{align}
The untruncated problem is denoted as follows
\begin{align}
 \mathcal{A}x = \mathcal{F}, \label{eqn:AxF2}
\end{align}
and the corresponding truncated problem reads
\begin{align}
 \hat{\mathcal{A}}x = \hat{\mathcal{F}}, \label{eqn:AxF}
\end{align}
where $\hat{\mathcal{A}}$ is given as follows
\begin{align}
\hat{\mathcal{A}} = \begin{pmatrix}
                \hat{A} & \hat{B}^* \\
		\hat{B} & -C  
               \end{pmatrix}, \quad \hat{\mathcal{F}} = 
               \begin{pmatrix}
                \hat{f} \\ g
               \end{pmatrix}, \quad \hat{x} = 
               \begin{pmatrix}
                \hat{u} \\ p
               \end{pmatrix},               
\end{align}  
where analogous to \eqref{eqn:ops}, we have following definitions for truncated operators
\begin{equation}
\begin{aligned} \label{eqn:ops_untrunc}
 \langle \hat{A}w, v \rangle &= \hat{a}(w,v), \quad \quad \hat{A} \in L(\hat{V}, \hat{V}^*), \\
 \langle \hat{B}w, q \rangle &= \hat{b}(w,q), \quad \quad \hat{B} \in L(\hat{V}, Q^*), \\
 \langle \hat{B}^*r, v \rangle &= \langle \hat{B}v, r \rangle, \quad \quad \hat{B}^* \in L(Q, \hat{V}^*),
\end{aligned}
\end{equation}
where $C$ is defined as in \eqref{eqn:ops}.
In \cite{Zulehner2011}, starting from the abstract theory on Hilbert spaces that
lead to representation of isometries, a preconditioner is proposed; it is based
on non-standard norms, or isometries that correspond to block diagonal
preconditioner of the following form
\begin{align}
  \mathcal{B} =
  \begin{pmatrix}
    \mathcal{I}_V & \\
    & \mathcal{I}_Q
  \end{pmatrix}.
\end{align}
In the next section, our goal is to determine $\mathcal{I}_Q$ and $\mathcal{I}_V.$
\subsection{\label{sec:norms} Choice of norm: a brief introduction to Zulehner's idea}
Before we move further, we introduce some notations.  The duality pairing
$\langle \cdot, \cdot \rangle_H$ on $H^* \times H$ is defined as follows
\begin{align*}
  \langle \ell, x \rangle_H = \ell(x) \quad \mbox{for all}~\ell \in H^*, \: x
  \in H.
\end{align*}
Let $\mathcal{I}_H:H \rightarrow H^*$ be an isometric isomorphism defined as
follows
\begin{align*}
  \langle \mathcal{I}_Hx, y \rangle = (x,y)_H.
\end{align*}
The inverse $\mathcal{R}_H = \mathcal{I}^{-1}_H$ is Riesz-isomorphism, by which
functionals in $H^*$ can be identified with elements in $H$ and we have
\begin{align*}
  \langle \ell, x \rangle_H = (\mathcal{R}_H\ell, x)_H.
\end{align*}
We already chose the type of norm in \eqref{eqn:normX}, we now look for explicit representation of
isometries or norms in terms of operators defined in \eqref{eqn:ops}. The main ingredient is the following theorem.
\begin{theorem}\cite{Zulehner2011}[p. 543, Th. 2.6]
  If there are constants $\underline{\gamma}_v, \bar{\gamma}_v,
  \underline{\gamma}_q, \bar{\gamma}_q > 0$ such that
 \begin{align}
   \underline{\gamma}_v \|w\|^2_V \leq a(w,w) + \| Bw \|^2_{Q^*} \leq
   \bar{\gamma}_v \| w\|^2_V, \quad \forall w \in V \label{eqn:infsup1}
 \end{align}
 and
 \begin{align}
   \underline{\gamma}_q \| r \|^2_Q \leq c(r,r) + \| B^*r \|^2_{V^*} \leq
   \bar{\gamma}_q \| r \|^2_Q, \quad \forall r \in Q \label{eqn:infsup2}
 \end{align}
then  
\begin{align} 
  \underline{c}_x \| z \|_X \leq \| \mathcal{A} z \|_{X^*} \leq \bar{c}_x
  \| z \|_X, \quad \forall z \in X \label{eqn:infsupA}
\end{align}
is satisfied with constants $\underline{c}_x, \bar{c}_x >0$ that depend
only on $\underline{\gamma}_v, \bar{\gamma}_v, \underline{\gamma}_q,$ and
on $
\bar{\gamma}_q.$ And, vice versa, if the estimates \eqref{eqn:infsupA} are
satisfied with constants $\underline{c}_x, \bar{c}_x > 0,$ then the
estimates \eqref{eqn:infsup1} and \eqref{eqn:infsup2} are satisfied.
\end{theorem}
Equivalently, as conjectured for \eqref{eqn:DoesThisHold1} and
\eqref{eqn:DoesThisHold2}, and recalling that $\hat{X} = \hat{V} \times Q
\subset V \times Q = X,$ we may ask whether the following bounds hold for
truncated system
\begin{align}
  \inf_{\hat{z} \in \hat{X}^*} \| \hat{\mathcal{A}}\hat{z}
  \|_{\hat{X}^*} \geq \inf_{z \in X^*} \| \mathcal{A}z \|_{X^*}, \quad \sup_{\hat{z} \in
    \hat{X}^*} \| \hat{\mathcal{A}}\hat{z}
   \|_{\hat{X}^*}
  \leq \sup_{z \in X^*} \| \mathcal{A}z \|_{X^*}.
\end{align}
However, we shall show a similar result in finite dimension
 using Fischer's theorem in Lemma \ref{lem:eigvals_truncated}.  In
\cite{Zulehner2011}, the terms $\| Bw \|^2_{Q^*}$ and $\| B^*r \|^2_{V^*}$ in
\eqref{eqn:infsup1} and \eqref{eqn:infsup2} respectively are defined using isometries $\mathcal{I}_V$ and $\mathcal{I}_Q$ as follows:
  \begin{align} \label{eqn:isom}
    \|B w\|^2_{Q^*} = \langle B^*\mathcal{I}^{-1}_QBw,w \rangle, \quad
    \|B^*r\|^2_{V^*} = \langle B\mathcal{I}^{-1}_VB^*r, r \rangle.
  \end{align}
Using \eqref{eqn:isom}, the equations \eqref{eqn:infsup1} and \eqref{eqn:infsup2} are equivalently written as
  follows
\begin{align*}
  \underline{\gamma}_v \langle \mathcal{I}_V w, w \rangle & \leq \langle (A + B^*
  \mathcal{I}^{-1}_Q B)w, w \rangle \leq
  \bar{\gamma}_v \langle \mathcal{I}_V w, w \rangle, \quad \mbox{for all}~w \in V, \\
  \underline{\gamma}_q \langle \mathcal{I}_Q r, r \rangle & \leq \langle (C + B
  \mathcal{I}^{-1}_V B^*)r, r \rangle \leq \bar{\gamma}_q \langle \mathcal{I}_Q
  r, r \rangle, \quad \mbox{for all}~r \in Q.
\end{align*}
In short, in new notation $\sim$ meaning ``spectrally similar'', we obtain the
following equivalent conditions for isometries
\begin{align*}
  \mathcal{I}_V & \sim A + B^*\mathcal{I}^{-1}_Q B \quad \mbox{and} \quad
  \mathcal{I}_Q \sim C + B \mathcal{I}^{-1}_V B^* \\
  \iff \mathcal{I}_V & \sim A + B^* (C + B \mathcal{I}^{-1}_V B^*)^{-1} B \quad
  \mbox{and} \quad \mathcal{I}_Q \sim C + B
  \mathcal{I}^{-1}_V B^*\\
  \iff \mathcal{I}_Q & \sim C + B (A + B^* \mathcal{I}^{-1}_Q B)^{-1} B \quad
  \mbox{and} \quad \mathcal{I}_V \sim A + B^* \mathcal{I}^{-1}_Q B
\end{align*}
Let $M$ and $N$ be any SPD matrices, consequently, they define inner products
and a Hilbert space structure in $\mathbb{R}^n.$ Moreover, the intermediate
Hilbert spaces between $M$ and $N$ are given as follows
\begin{align*}
  [M,N]_\theta=M^{1/2}(M^{-1/2}NM^{-1/2})^\theta M^{1/2}, \quad \theta \in
  [0,1].
\end{align*}
Continuing from above, when $A$ and $C$ are non singular, the 
generic form of the norms are given by the following lemma.
\begin{lemma}
  \label{I_V_and_I_Q}
  Let $A$, consequently, $C=\eta A, \, \eta >0$ be nonsingular. Then 
  \begin{align} \label{eqn:templ}
    \mathcal{I}_V = A + [A, B^TC^{-1}B]_\theta, \quad \mathcal{I}_Q = C + [C,
    BA^{-1}B^T]_{1-\theta}, \quad \theta \in [0,1].
  \end{align}
\end{lemma}
\begin{proof}
  See \cite{Zulehner2011}[p. 547-548].
\end{proof}
The isometries $\mathcal{I}_V$ and $\mathcal{I}_Q$ above provide a general template for obtaining a variety of preconditioners. Obviously, our goal is to find those that are easier to compute with numerically.
Before we propose preconditioners, we shall need some properties of the $(1,1)$
block of $\mathcal{A},$ and that for the negative Schur complement $S = C +
\hat{B} \hat{A}^{-1}\hat{B}^T.$ Such properties will be useful in developing preconditioners 
using $\mathcal{I}_V$ and $\mathcal{I}_Q.$

\subsection{Properties of the system matrix and Schur complement}
An important property that we shall need shortly when analyzing preconditioners
is that the eigenvalues of the truncated matrix is bounded from above and below
by the eigenvalues of the untruncated matrix.
\begin{lemma}
\label{ind_A}
The operator $\mathcal{A}$ is symmetric and indefinite.
\end{lemma}
\begin{proof}
Symmetry is obvious. Indefiniteness follows from below:
\begin{align}
	x^T \mathcal{A} x = \left[\begin{array}{c}
		u^T \:  v^T	
	\end{array}\right]
	\left[\begin{array}{cc}
		\bar{K} & M \\
		M & -\eta \bar{K}
	\end{array}\right] 
	\left[ \begin{array}{c}
		u \\ v
	\end{array}\right] = \| u \|_{\bar{K}}^2 - \eta \| v \|_{\bar{K}}^2 + 2 Re |(Mu,v)|.
\end{align}
For the choice of $[u^T, v^T] = [0^T, v^T], \: x^T \mathcal{A} x 
	= -\eta \| v \|^2_{\bar{K}} \leq 0.$
\end{proof}
\begin{lemma}
  \label{spd}
  $A, A + B$ is SPD.
\end{lemma}
\begin{proof}
  From \eqref{eqn:not}, we recall that $A = \bar{K} = K + mm^T.$ Here $K$ being a
  stiffness matrix corresponding to natural boundary condition is SPD except on
  the span of vector $\mathbf{1} = [1,1,1,\dots,1]^T,$ which is in the kernel of
  $K,$ but $(mm^T\mathbf{1},\mathbf{1})>0.$ Also,  $B = M$ being a mass
  matrix is SPD, $A + B$ is SPD.
\end{proof}

\begin{fact}[{\bf Permutation preserves eigenvalues}]
  \label{lem:perm_eig}
  Let $P \in \mathbb{Z}^{n \times n}$ be a permutation matrix, then $P^T \hat{A}
  P$ and $\hat{A}$ are similar.
\end{fact}
\begin{proof}
  $P$ being a permutation matrix, $P^TP = Id,$ hence the proof.
\end{proof}
\begin{lemma}[{\bf Poincare separation theorem for eigenvalues}]
  \label{lem:poincare}
  Let $Z \in \mathbb{R}^{n \times n}$ be any symmetric matrix with eigenvalues
  $\lambda_1 \leq \lambda_2 \leq \cdots \leq \lambda_n,$ and let $P$ be a
  semi-orthogonal $n \times k$ matrix such that $P^TP = Id \in \mathbb{R}^{k
    \times k}.$ Then the eigenvalues $\mu_1 \leq \mu_2 \leq \cdots \leq \mu_{n-k+i}$ of
  $P^TZP$ are separated by the eigenvalues of $Z$ as follows
  \begin{align}
    \lambda_i \leq \mu_i \leq \lambda_{n-k+i}.
  \end{align}
\end{lemma}  
\begin{proof}  
  The theorem is proved in \cite[p. 337]{Rao1998}.
\end{proof}
\begin{lemma}[{\bf Eigenvalues of the truncated (1,1) block}]
  \label{thm:eig_bound}
  Let $n = |\mathcal{N}_h|.$ Let $\lambda_1 \leq \lambda_2 \leq \dots \leq
  \lambda_n$ be the eigenvalues of $A,$ and let $\hat{\lambda}_1 \leq
  \hat{\lambda}_2 \leq \dots \leq \hat{\lambda}_n$ be the eigenvalues of
  truncated matrix $\hat{A}.$ Let $k=\sum_{i=1}^n T(i,i)$ be the number of
  untruncated rows in $\hat{A}.$ Let $\hat{\lambda}_{n_1} \leq
  \hat{\lambda}_{n_2} \leq \dots \leq \hat{\lambda}_{n_k}$ be the
  eigenvalues of $\hat{A}$ excluding the $n-k$ trivial eigenvalues one of
  $\hat{A}$ that appear due to addition of $\hat{T}$ in \eqref{eqn:not}. Then the
  eigenvalues of truncated and untruncated matrices are related as follows
  \begin{align}
    \lambda_i \leq \hat{\lambda}_{n_i} \leq \lambda_{n-k+i}, \quad 1 \leq i \leq k.
  \end{align}
\end{lemma}
\begin{proof}
  The proof follows by an application of Poincare separation theorem. For this,
  it is convenient to permute the matrix into truncated and untruncated rows and
  columns. Let $P$ be a permutation matrix that renumbers the rows and columns
  such that the truncated rows and columns are numbered first, i.e.,
\begin{align}
  P^T\hat{A}P = \begin{pmatrix}
    I & \\
      & R^TP^T\hat{A}PR
  \end{pmatrix},
\end{align}
where $R \in \mathbb{R}^{n \times k}$ is the restriction matrix defined as
follows
\begin{align}
  R =
  \begin{pmatrix}
    \begin{pmatrix*}[c]
      0 & 0 & \dots & 0 \\
      \vdots & \hdots & \hdots & 0 \\
      0 & 0 & \dots & 0
    \end{pmatrix*}_{n-k \times k} \\    
    \begin{pmatrix*}[c]
      1 & 0 & \dots & 0 \\
      0 & 1 & \dots & 0 \\
      \vdots & \ddots & \dots & 0 \\
      0 & 0 & \dots & 1
    \end{pmatrix*}_{k \times k}
  \end{pmatrix}. \label{eqn:R}
\end{align}
Clearly $R^T R = Id \in \mathbb{R}^{k \times k}.$ From Lemma \ref{lem:perm_eig},
$P^T \hat{A} P$ and $\hat{A}$ are similar and $P^TAP$ and $A$ are
similar. Applying Lemma \ref{lem:poincare}, to $P^TAP$ and $R^T(P^TAP)R,$ we
have the proof.
\end{proof}
\begin{corollary} \label{cor:eig_trunc_untrunc}
  From Theorem \ref{thm:eig_bound}, we have $$\lambda_{\min}(\hat{A}) \geq
  \lambda_{\min}(A)>0,$$ 
$$\lambda_{\max}(\hat{A}) \leq
  \lambda_{\max}(A),$$
hence $\hat{A}$ is SPD since $A$ is SPD from Lemma \eqref{spd}.
  Moreover, $cond(\hat{A}) \leq cond(A).$
\end{corollary}
\begin{remark}
From  \eqref{eqn:TandThat} and \eqref{eqn:trun_simplified}, we have 
\begin{align}
\hat{\mathcal{A}} = \begin{pmatrix}
T & \\
  & Id
\end{pmatrix} \begin{pmatrix}
A & B \\ B^T & -\eta A
\end{pmatrix} \begin{pmatrix}
T & \\ & Id
\end{pmatrix}.
\end{align}
Using similar argument as in Lemma \eqref{thm:eig_bound} and Cor. \ref{cor:eig_trunc_untrunc}, we have 
$\lambda_{\min}(\hat{\mathcal{A}}) \geq
  \lambda_{\min}(\mathcal{A})>0$ and 
$\lambda_{\max}(\hat{\mathcal{A}}) \leq
  \lambda_{\max}(\mathcal{A}).$
\end{remark}
We know that the matrix $M$ is SPD, and $K$ is a SPSD.  In the following, we
observe the properties of truncated matrices.
\begin{definition}
  Let $G(A)=(V,E)$ be the adjacency graph of a matrix $A \in \mathbb{R}^{N
    \times N}$.  The matrix $A$ is called irreducible if any vertex $i \in V$ is
  connected to any vertex $j \in V$. Otherwise, $A$ is called reducible.
\end{definition}
\begin{definition} \label{def:mmatrix1} A matrix $A \in \mathbb{R}^{N \times N}$
  is called an $M$-matrix if it satisfies the following three properties:
  $a_{ii}>0$ for $i=1,\dots,N,$
  $a_{ij}\le 0$ for $i \neq j$, $i,j=1,\dots,N,$ and
  $A$ is non-singular and $A^{-1} \ge 0.$
\end{definition}
\begin{definition}
  A square matrix $A$ is strictly diagonally dominant if the following holds
  \begin{align}
  |a_{ii}| > \sum_{j \neq i}|a_{ij}|, \quad i=1,\dots,N, \label{eqn:diag_dom}
  \end{align}
  and it is called irreducibly diagonally dominant if $A$ is irreducible and the
  following holds
  \begin{align}
    |a_{ii}| \geq \sum_{j \neq i}|a_{ij}|, \quad i=1,\dots,N, \label{eqn:irr_diag_dom} 
  \end{align}
  where strict inequality holds for at least one $i$.
\end{definition}
A simpler criteria for $M$-matrix property is then given by the following theorem.
\begin{lemma} \label{thm:mmatrix} If the coefficient matrix $A$ is strictly or
  irreducibly diagonally dominant and satisfies the following conditions
  \begin{enumerate}
  \item $a_{ii}>0$ for $i=1,\dots,N$
  \item $a_{ij}\le 0$ for $i \neq j$, $i,j=1,\dots,N$
  \end{enumerate} then $A$ is an $M-$matrix.
\end{lemma}
 
\begin{remark} \label{rem:KnotM} Note that $K$ is not an $M$-matrix because $K
  \cdot \mathbf{1}=0,$ hence, the condition of \eqref{eqn:irr_diag_dom} that
  strict inequality must hold for atleast one row is not satisfied. Moreover,
  mass matrix $M$ has positive off-diagonal entries, hence, it does not satisfy
  the hypothesis of Lemma \ref{thm:mmatrix}, thus, we cannot conclude that $M$
  is an $M$-matrix either.
\end{remark}
Although, from Lemma \ref{rem:KnotM}, $K$ is not an $M$-matrix, the truncated
matrix $\hat{K}$ defined in \eqref{eqn:not} with at least one truncated row
and column is an $M$-matrix. Let the set of truncated nodes be defined by
\begin{align}
\mathcal{N}^{\bullet}_h = \left\{i \: : \: T(i,i) = 0 \right\}.
\end{align}
\begin{lemma} \label{lem:mmatrices}
  Let $|\mathcal{N}^{\bullet}_h|\geq 1,$ then $\hat{K},$ $P^T \hat{K} P,$ and $R^T
  P^T \hat{K} P R$ are $M$-matrices.
\end{lemma}
\begin{proof}
  Since $|\mathcal{N}^{\bullet}_h|\geq 1,$ for all rows corresponding to truncated set
  $\mathcal{N}^{\bullet}_h,$ it is trivial that we have strict diagonal dominance:
  \begin{align}
    \hat{k}_{ii} = 1 = |\hat{k}_{ii}| > 0 = \sum_{j \neq i}\hat{k}_{ij}, \quad
    \forall i \in \mathcal{N}^{\bullet}_h, \quad j=1,\dots,
    |\mathcal{N}_h|, \label{eqn:strict_diag_dom}
  \end{align}
  where as, for rows corresponding to untruncated set $\mathcal{N}_h \setminus
  \mathcal{N}^{\bullet}_h,$ we have
  \begin{align}
    \hat{k}_{ii} =  k_{ii} = |\hat{k}_{ii}| \geq \sum_{j \neq
      i} |k_{ij}| \geq \sum_{j \neq i}|\hat{k}_{ij}|, \quad \forall i
    \in \mathcal{N}_h \setminus \mathcal{N}^{\bullet}_h, \quad j=1,\dots,
    |\mathcal{N}_h|. \label{eqn:dd}
  \end{align}
  Moreover, we have 
  \begin{align} \label{eqn:entries}
    \hat{k}_{ij} = 
    \begin{cases}(\mbox{when}~i=j) \quad
      \begin{cases}
        1, \quad \forall i \in \mathcal{N}^{\bullet}_h, \\
        k_{ii} > 0, \quad \forall i \in \mathcal{N}_h \setminus
        \mathcal{N}^{\bullet}_h,
      \end{cases}\\
      (\mbox{when}~i \neq j) \quad k_{ij} < 0, \quad \forall i\in\mathcal{N}_h.
    \end{cases}
  \end{align}
  The sufficient conditions of Lemma \ref{thm:mmatrix} are now satisfied: from
  \eqref{eqn:strict_diag_dom} and \eqref{eqn:dd}, we conclude that $\hat{K}$
  is irreducibly diagonaly dominant, and \eqref{eqn:entries} satisfies
  hypothesis 1. and 2. of Lemma \ref{thm:mmatrix}. Hence, $\hat{K}$ is an
  $M$-matrix.  $P^T\hat{K}P$ being the symmetric permutation of rows and
  columns of $\hat{K}$ remains an $M$-matrix. Lastly, $R^T P^T \hat{K} P
  R$ being a principle submatrix of $P^T \hat{K} P$ is also an $M$-matrix,
  see proof in \cite{Horn1991}[p. 114].
\end{proof}
\begin{remark}
  To solve with $\hat{A},$ we use the Sherman-Woodbury
  formula
\begin{align}
  \hat{A}^+ = (\hat{K} + \tilde{m}\tilde{m}^T)^{+} = \hat{K}^{+} -
  \frac{\hat{K}^{+}\tilde{m}\tilde{m}^T \hat{K}^{+}}{1 +
    \tilde{m}^T\hat{K}^{+} \tilde{m}}.
\end{align}
Here $\hat{K}^+$ denotes pseudo-inverse of $\hat{K},$ however, $\hat{K}$ is a non-singular $M$-matrix for
$|\mathcal{N}_h^\bullet|\geq 1,$ thus, in this case, we may replace $\hat{K}^+$ by
$\hat{K}^{-1}.$ Since $\hat{K}$ is an $M$-matrix from Lemma \ref{lem:mmatrices} above for $|\mathcal{N}_h^\bullet|\geq 1,$ algebraic multigrid,
or incomplete Cholesky (which is as stable as exact Cholesky factorization,
\cite{Meijerink1977}[Theorem 3.2] ) may be used as a preconditioner to solve
with $\hat{K}$ inexactly.
\end{remark}   
Before we define a preconditioner involving Schur complement, it is essential to know whether $S$ is
nonsingular.

In the following, we provide a slightly
different proof then in \cite{Graeser2009}, where similar result is shown for
continuous Schur complement.
\begin{theorem}
  \label{schur_spd}
  The negative Schur complement $S = C +
  \hat{B}\hat{A}^{-1}\hat{B}^T$ is non-singular, in particular, SPD
  if and only if $|\mathcal{N}_h^{\bullet}| < |\mathcal{N}_h|.$
\end{theorem}
\begin{proof}
  If $|\mathcal{N}_h^{\bullet}|= |\mathcal{N}_h|,$ then $\hat{B}$ is the
  zero matrix, consequently, $S = C = \eta K$ is singular since $K$ corresponds
  to stiffness matrix with pure Neumann boundary condition.  For other
  implication, we recall from \eqref{eqn:not} that $\hat{B}^T = \hat{M}^T = TM,$ where $T$
  is defined in \eqref{eqn:TandThat}. The $(i,j)^{th}$ entry of element mass
  matrix is given as follows
  \begin{align}
    M^K_{ij} = \int_K \phi_i \phi_j dx = \frac{1}{12} (1 + \delta_{ij}
    |K|) \label{eqn:element_mass} \quad i,j=1,2,3,
  \end{align}
  where $\delta_{ij}$ is the Kronecker symbol, that is, it is equal to 1 if
  $i=j,$ and 0 if $i \neq j.$ Here $\phi_1, \phi_2,$ and $\phi_3$ are hat
  functions on triangular element $K$ with local numbering, and $|K|$ is the
  area of triangle element $K.$ From \eqref{eqn:element_mass}, it is easy to see
  that
  \begin{align} \label{eqn:element_mass3} M^K = \frac{1}{12} \begin{pmatrix}
      2 & 1 & 1 \\
      1 & 2 & 1 \\
      1 & 1 & 2
    \end{pmatrix}.
  \end{align}
  Evidently, entries of global mass matrix $M = \sum_K M^K$ are also all
  positive, hence all entries of truncated mass matrix $\hat{M}$ remain
  non-negative. In particular, due to our hypothesis
  $|\mathcal{N}^{\bullet}|>0,$ there is atleast one untruncated column, hence,
  atleast few positive entries. Consequently, $M \mathbf{1} \neq 0,$ i.e.,
  $\mathbf{1}$ or $\text{span}\{ \mathbf{1} \}$ is neither in kernel of $M,$ nor
  in the kernel of $\hat{M},$ in particular, $\mathbf{1}^T \hat{M}^T
  \mathbf{1} > 0.$ The proof of the theorem then follows since $C$ is SPD except
  on $\mathbf{1}$ for which $\hat{B}^T \mathbf{1}$ is non-zero, and the fact
  that $\hat{A}$ is SPD yields
  \begin{align}
    \left \langle \hat{B}\hat{A}^{-1}\hat{B}^T \mathbf{1}, \mathbf{1} \right
    \rangle = \left \langle \hat{A}^{-1}(\hat{B}^T \mathbf{1}),
      (\hat{B}^T\mathbf{1}) \right \rangle = \left \langle
      \hat{A}^{-1}(-\hat{M}^T \mathbf{1}), (-\hat{M}^T\mathbf{1}) \right \rangle
    > 0.
  \end{align}
\end{proof}

\begin{remark}
  \label{schur_spd}
  The negative Schur complement $S = \eta \hat{A} + \hat{B}^T
  \hat{A}^{-1} \hat{B}$ with $\hat{\mathcal{A}}$ defined in
  \eqref{eqn:trun_simplified} is nonsingular even for $|\mathcal{N}_h| =
  |\mathcal{N}|.$
\end{remark}


\begin{theorem}[{\bf Condition number of the truncated Schur complement}]
Following holds
\begin{itemize}
\item $\lambda_{\min} (\hat{S}) > \lambda_{\min} (\mathcal{A})$ and
  $\lambda_{\max} (\hat{S}) < \lambda_{\max} (\mathcal{A})$
\end{itemize}
\end{theorem}
\begin{proof}
From \cite[p. 111]{smith1996}, following holds
\begin{align*}
	\lambda_{\min}(\hat{S}) > \lambda_{\min}
        (\hat{\mathcal{A}}), \quad \lambda_{\max} (\hat{S}) 
        < \lambda_{\max}(\hat{\mathcal{A}}),
\end{align*}
and from Poincare separation theorem \ref{thm:eig_bound}, we have
\begin{align*}
  \lambda_{\min}(\hat{\mathcal{A}}) > \lambda_{\min} (\mathcal{A}), \quad
  \lambda_{\max} (\hat{\mathcal{A}}) < \lambda_{\max} (\mathcal{A}).
\end{align*}
\end{proof}

\section{\label{sec:precon}Preconditioner for the Linear System}
In this section, we propose preconditioners for the linear system for the
untruncated system, and we propose the related truncated
preconditioners for the truncated system.
\subsection{\label{sec:precI}Block Diagonal Preconditioner (BD)} 
Since $A,$ hence, $C=\eta A, \eta >0$ are non-singular, assumptions of Lemma
\ref{I_V_and_I_Q} are satisfied.  Specifically, $\theta=1/2$ yields
\begin{align*}
  \mathcal{I}_V = A + \eta^{-1/2}[A,
  BA^{-1}B]_{1/2}, \quad \mathcal{I}_Q = C +
  \eta^{1/2}[A, BA^{-1}B]_{1/2}.
\end{align*}
But $[A, BA^{-1}B]_{1/2} = B,$ thus, further
simplification yields
\begin{align}
  \mathcal{I}_V = A + \eta^{-1/2}B, \quad \mathcal{I}_Q = \eta A +
  \eta^{1/2}B. \label{eqn:pre2}
\end{align} 
Choice of $\theta=0,1$ in \eqref{eqn:templ} in Lemma \ref{I_V_and_I_Q} brings back Schur complements, but, we have avoided
it. However, later we shall consider the case $\theta=0.$ Any other intermediate value of $\theta$ does not look interesting or useful. For large problems, it won't be feasible to solve with $\mathcal{I}_V$ and
$\mathcal{I}_Q$ in \eqref{eqn:pre2} exactly, or not even up to double precision
using prohibitively expensive direct methods such as QR or LU factorizations
\cite{Golub1996}.
\begin{remark}[{\bf Ensuring $M$-matrix property of the preconditioner}]
For
existence and subsequent application of fast inexact solvers for $\mathcal{I}_V$ and
$\mathcal{I}_Q,$ an important property to look for
is $M$-matrix property, but it must be pointed out that this property is not guaranteed in
\eqref{eqn:pre2}, consequently, the diagonal dominance of $\mathcal{I}_V$ or
$\mathcal{I}_Q$ may be lost for certain values of $\eta.$ To sketch the proof
for $I_Q$, we observe that
\begin{align}
  A^K_{ij} & = \left( \int_K \nabla \phi_i \cdot \nabla \phi_j dx
    + \int_K \phi_i dx
    \int_K \phi_j dx \right), \quad i,j=1,2,3, \\
                 & = (b_ib_j + c_ic_j) \int_K dx + mm^T = (b_ib_j + c_ic_j) |K| + mm^T,
  \quad i,j=1,2,3,
\end{align}
where $|K|$ is the area of triangle element $K,$ and 
\begin{align}
b_i = \frac{x_2^j - x_2^k}{2|K|}, \quad c_i = \frac{x_1^k - x_1^j}{2|K|}, \quad \{ j,k\} \in \{ 1,2,3\},
\end{align}
where $(x_1^i, x_2^i), i=1,2,3$ are coordinates of the three vertices of element $K.$
The $(i,j)^{th}$ entry of element mass matrix is given in 
\eqref{eqn:element_mass}.
  Evidently, entries of global mass matrix $M = \sum_K M^K$ are also all
  positive.
We have
\begin{align}
  \eta A^K_{ij} + \eta^{1/2}B^K_{ij} = \eta A^K_{ij} +
  \eta^{1/2}M^K_{ij} = \eta (b_ib_j + c_ic_j)|K| 
 + \eta mm^T +  \eta^{1/2}\frac{1}{12}(1 + \delta_{ij})|K|.
\end{align}
Thus, the the off-diagonal entries of $\mathcal{I}_Q$ may become positive,  due to addition of the mass matrix
$M$ for certain values of $\eta,$ thereby voilating the sufficient condition of Lemma  \ref{thm:mmatrix} for $\mathcal{I}_Q$ to be an $M$-matrix. However, the $M$-matrix property of $\mathcal{I}_V$ is
ensured by lumping the mass matrix: we proved earlier in Lemma \ref{lem:mmatrices} that the truncated matrix
$\hat{K}$ is an M-matrix if there is at least one truncated node, addition of lumped mass matrix further enhances the diagonal dominance of $\mathcal{I}_Q,$ and does not violate sufficient condition of Lemma  \ref{thm:mmatrix}. Similarly, $\mathcal{I}_V = \eta \mathcal{I}_Q$ can
be kept $M$-matrix. Hence, algebraic multigrid may be used to solve with
$\mathcal{I}_V$ and $\mathcal{I}_Q.$
\end{remark}
 The following
eigenvalue bound is similar to the one in \cite{Zulehner2011}. Our system matrix
is different in that in place of $K$ we have $M$ and in place of $M$ we have
$\bar{K}.$ Consequently, for our system, the bound is slightly tighter in the sense that
the eigenvalues lie in the open interval as shown below, whereas, in
\cite{Zulehner2011} they lie in the closed interval.
\begin{theorem}[{\bf Eigenvalue bound for $\mathcal{B}_{\tt bd}^{-1} \mathcal{A}$}]
\label{th:eig_bound_blk_diag}
There holds
\begin{align}
   \lambda \left(\left[  
      \begin{array}{cc}
        \bar{K} + \eta^{-1/2}M & 0 \\
        0 & \eta \bar{K} + \eta^{1/2}M
      \end{array}
    \right]^{-1} \left[ 
      \begin{array}{cc}
        \bar{K} & M \\
        M & -\eta \bar{K}  
      \end{array}
    \right] \right) \in \left(-1, -\frac{1}{\sqrt{2}} \right) \cup \left( \frac{1}{\sqrt{2}}, 1 \right).
\end{align}
\end{theorem}
\begin{proof}
We first consider the generalized eigenvalue problem
\begin{align} 
\bar{K}z = \mu (\bar{K} + \eta^{-1/2}M )z. \label{eqn:gen_eig} 
\end{align}
Since $\bar{K}$ and $\bar{K} +
\eta^{-1/2}M$ are SPD, there is a basis ${e_1, e_2,
  \dots, }$ of eigenvectors $e_i$ with corresponding eigenvalues $\mu_i \in
(0,1),$ which are orthonormal with respect to the $\bar{K} + \eta^{-1/2}M$
inner product. This is easily seen by looking at the Rayleigh quotient
\begin{align}
  0 < \frac{ \langle x, x \rangle_{\bar{K}}}{ \langle x, x \rangle_{\bar{K} +
      \eta^{-1/2} M}} = \frac{x^T \bar{K} x}{x^T \bar{K} x + \eta^{-1/2} x^T M x } < 1
  \quad \forall x \: ~\text{s.t.}~ \: \| x \|_{\bar{K} + \eta^{-1/2} M}=1,
\end{align}
since, $x^T\bar{K}x + \eta^{-1/2}x^TMx \geq x^T\bar{K}x, \forall x.$
We now look at the following generalized eigenvalue problem
\begin{align}
 \label{eqn:gen_eig2}
\left[ 
      \begin{array}{cc}
        \bar{K} & M \\
        M & -\eta \bar{K}  
      \end{array}
    \right] \left[\begin{array}{c}
      u \\ v
      \end{array} \right] = \lambda \left[  
      \begin{array}{cc}
        \bar{K} + \eta^{-1/2}M & 0 \\
        0 & \eta \bar{K} + \eta^{1/2}M
      \end{array}
    \right] \left[\begin{array}{c}
      u \\ v
      \end{array} \right].
\end{align}
Since $\{ e_1, e_2, \dots \}$ are the eigenbasis, we have
\begin{align}
  u = \sum_i \hat{u}_ie_i, \quad v = \sum_i \hat{v}_ie_i. \label{eqn:uAndv}
\end{align}
Substituting $u$ and $v$ from \eqref{eqn:uAndv} in \eqref{eqn:gen_eig2}, and
looking at the $i$th rows of both equations of \eqref{eqn:gen_eig2}, we have
\begin{align}
  \hat{u}_i \bar{K}e_i + \hat{v}_i M e_i &= \lambda \hat{u}_i (\bar{K} +
  \eta^{-1/2} M) e_i, \label{eqn:eig1} \\
  \hat{u}_i M e_i - \eta \hat{v}_i \bar{K}e_i &= \lambda \hat{v}_i \eta (\bar{K} +
  \eta^{-1/2} M)e_i. \label{eqn:eig2}
\end{align}
In \ref{eqn:gen_eig}, choosing $z=e_i,$ and multiplying by $e_i^T$ from the left, we have
\begin{align}
  \mu_i = \frac{e_i^T \bar{K} e_i}{e_i^T (\bar{K} + \eta^{-1/2} M) e_i}. \label{eqn:mui}
\end{align}
Using \eqref{eqn:gen_eig}, for $z=e_i,$ $Me_i$ reads
\begin{align}
  Me_i = \eta^{1/2} (1 - \mu_i) \mu_i^{-1} \bar{K}e_i. \label{eqn:MinK}  
\end{align}
Multiplying \eqref{eqn:eig1} and
\eqref{eqn:eig2} on the left by $e_i^T,$ then dividing both equations by  $e_i^T(\bar{K} + \eta^{-1/2} M) e_i > 0,$ then substituting $Me_i$ from \eqref{eqn:MinK} in \eqref{eqn:eig1} and
\eqref{eqn:eig2}, we have
\begin{align}
  \mu_i \hat{u}_i + \eta^{1/2}(1 - \mu_i) \hat{v}_i &= \lambda
  \hat{u}_i, \\
  \eta^{1/2}(1 - \mu_i) \hat{u}_i - \eta \mu_i \hat{v}_i &= \lambda \eta \hat{v}_i, 
\end{align}    
which in matrix form reads
\begin{align}
  \left[ \begin{array}{cc}
      \mu_i & \eta^{1/2}(1 - \mu_i) \\
      \eta^{1/2}(1 - \mu_i) & -\eta \mu_i
      \end{array}
    \right] \left[ 
      \begin{array}{c}
        \hat{u}_i \\ \hat{v}_i
        \end{array}
    \right] = \lambda \left[   
      \begin{array}{cc}
        1 & \\
         & \eta
        \end{array}
    \right] \left[
      \begin{array}{c}
        \hat{u}_i \\ \hat{v}_i
      \end{array}
      \right].
\end{align}
Since $[u,v]$ is an eigenvector, there exists at least one $i$ s.t. $\hat{u}_i,
\hat{v}_i \neq 0,$ for which following holds
\begin{align}
  \text{det} \left( \left[ \begin{array}{cc}
      \mu_i & \eta^{1/2}(1 - \mu_i) \\
      \eta^{1/2}(1 - \mu_i) & -\eta \mu_i
      \end{array}
    \right] - \lambda \left[   
      \begin{array}{cc}
        1 & \\
         & \eta
        \end{array}
    \right]\right) &= 0, \\
\implies -\eta (\mu_i - \lambda) (\mu_i + \lambda) - \eta (1 - \mu_i)^2 &= 0, \\
\implies \lambda^2 - \mu_i^2 - (1-\mu_i^2) &= 0.
\end{align}
We have $|\lambda| = \sqrt{\mu_i^2 + (1- \mu_i)^2}$ which monotonically
decreases for $\mu_i \in (0, 1/2],$ and monotonically increases for
$\mu_i \in [1/2, 1).$ In particular, $|\lambda|=1/\sqrt{2}$ is minimum value at $\mu_i =
1 / 2,$ and has a maxima for either $\mu_i = 0$ or for $\mu_i = 1.$ Hence
$|\lambda| \in [1/\sqrt{2}, 1).$
\end{proof}
The following Lemma shows that the condition number is of the order one. From
Lemma above, the eigenvalues of $\mathcal{B}_{\tt bd}^{-1} \mathcal{A}$ may be
negative; the condition number in this case is defined as the ratio of modulus
of maximum and minimum eigenvalues, i.e.,
\begin{align}
  \kappa(\mathcal{B}_{\tt bd}^{-1}\mathcal{A}) = \frac{\max
    \left(|\lambda(\mathcal{B}_{\tt bd}^{-1}\mathcal{A})| \right)}{\min
    \left(|\lambda(\mathcal{B}_{\tt bd}^{-1}\mathcal{A}| \right)}.
\end{align}
\begin{corollary}[{\bf Condition number estimate for $\mathcal{B}_{\tt bd}^{-1} \mathcal{A}$}]
  \label{cond}
  The condition number is given as follows
  \begin{align}
       \kappa \left(\left[  
      \begin{array}{cc}
        \bar{K} + \eta^{-1/2}M & 0 \\
        0 & \eta \bar{K} + \eta^{1/2}M
      \end{array}
    \right]^{-1} \left[ 
      \begin{array}{cc}
        \bar{K} & M \\
        M & -\eta \bar{K}  
      \end{array}
    \right] \right) < \sqrt{2}.
   \end{align}
\end{corollary}
\begin{proof} 
  Follows from Theorem \ref{th:eig_bound_blk_diag}. 
\end{proof}
Our goal is to solve truncated problem. We want to bound the
extreme eigenvalues of the preconditioned truncated matrix by those for the preconditioned untruncated
matrix. To this end, following theorem is useful.
\begin{lemma}{\bf (Fischer)}[p. 281, \cite{Stewart1990}]
\label{lem:fischer}
Let $X = \mathbb{R}^n,$ where $n$ is some positive integer. Let $A, B \in
\mathbb{R}^{n \times n}$ be any two Hermatian matrices and let $B$ be SPD. Let
the eigenvalues of $B^{-1}A$ be ordered as follows $\lambda_{\max} = \lambda_1
\geq \lambda_2 \geq \cdots \geq \lambda_n = \lambda_{\min}$ (Note: such an
ordering is possible because the eigenvalues of $B^{-1}A$ are real, since $B^{-1}A$ is
similar to a symmetric matrix $B^{-1/2}A B^{-1/2}$).  Then
\begin{align*}
  \lambda_i = \max_{\dim(X)=i} \: \underset{x \neq 0}{\min_{x \in X}} \:
  \frac{x^T A x}{x^T B x},
\end{align*}
and
\begin{align*}
  \lambda_i = \max_{\dim(X)= n-i+1} \: \underset{x \neq 0}{\max_{x \in X}} \:
  \frac{x^T A x}{x^T B x}.
\end{align*}
In particular, there holds
\begin{align}
  \lambda_{\min} \, (B^{-1}A) = \underset{x \in X}{\min_{x \neq 0}} \: \frac{x^T
    A x}{x^T B x}, \quad \lambda_{\max}(B^{-1}A) = \underset{x \in X}{\max_{x
      \neq 0}} \: \frac{x^T A x}{x^T B x}.
\end{align}
\end{lemma}
\begin{lemma}[{\bf Bound on extreme eigenvalues of the preconditioned truncated 
    matrix}]
  \label{lem:eigvals_truncated}
  The non-zero extreme eigenvalues of the preconditioned truncated  operator with
  atleast one truncation are bounded from above and below by the eigenvalues of
  the preconditioned untruncated  operator.
\end{lemma}
\begin{proof}
  Let $P$ be a permutation matrix that permutes the rows and columns such that
  the truncated nodes are numbered first. Let $T$ be the truncation matrix as in
  \eqref{eqn:TandThat}, and let $R$ be a restriction operator as in
  \eqref{eqn:AxF2} that compresses the matrix to untruncated nodes. We use
  the following notation for compressed matrices
  \begin{align}
    \widecheck{\mathcal{A}} = R^T T P^T \mathcal{A} P T R, \quad
    \widecheck{\mathcal{B}}_{\tt bd} = R^T T P^T \mathcal{B}_{\tt bd} P T R.
  \end{align}
  Let $Z = PTR.$ To use Lemma \ref{lem:fischer}, we note that $\mathcal{B}_{\tt
    bd}$ is SPD, hence from Poincar\'e separation theorem, i.e., from Lemma
  \ref{lem:poincare}, $\widecheck{\mathcal{B}}_{\tt bd}$ is SPD. Alternatively, $\widecheck{\mathcal{B}}_{\tt bd}$ being a principle submatrix of $P^T\mathcal{B}_{\tt bd}P$ is SPD.  Since
  $\mbox{Dim(Range}(TR)) \leq |\mathcal{N}_h|,$ we clearly have
  \begin{align*}
    \lambda_{\mbox{min}} (\widecheck{B}_{\tt bd}^{-1}\widecheck{A}) = \underset{x_z \in
      \mathbb{R}^{|\mathcal{N}_h|}}{\min_{x_z = Zx \neq 0}} \frac{x^T
      \widecheck{\mathcal{A}} x}{x^T \widecheck{\mathcal{B}}_{\tt bd} x} \geq \underset{x
      \in \mathbb{R}^{|\mathcal{N}_h|}}{\min_{x \neq 0}} \frac{(Px)^T
      \mathcal{A} (Px)}{(Px)^T \mathcal{B}_{\tt bd} (Px)} = \underset{x \in
      \mathbb{R}^{|\mathcal{N}_h|}}{\min_{x \neq 0}} \frac{x^T \mathcal{A}
      x}{x^T \mathcal{B}_{\tt bd} x} = \lambda_{\mbox{min}}
    (\mathcal{B}_{\tt bd}^{-1}\mathcal{A}).
  \end{align*}
  Similarly, we have
  \begin{align*}
    \lambda_{\mbox{max}} (\widecheck{B}_{\tt bd}^{-1}\widecheck{A}) = \underset{x_z \in
      \mathbb{R}^{|\mathcal{N}_h|}}{\max_{x_z = Zx \neq 0}} \frac{ x^T
      \widecheck{\mathcal{A}} x}{x^T \widecheck{\mathcal{B}}_{\tt bd} x} \leq \underset{x
      \in \mathbb{R}^{|\mathcal{N}_h|}}{\max_{x \neq 0}} \frac{(Px)^T
      \mathcal{A} (Px)}{(Px)^T\mathcal{B}_{\tt bd} (Px)} = \underset{x \in
      \mathbb{R}^{|\mathcal{N}_h|}}{\max_{x \neq 0}} \frac{x^T \mathcal{A}
      x}{x^T \mathcal{B}_{\tt bd} x} = \lambda_{\mbox{max}}
    (\mathcal{B}_{\tt bd}^{-1}\mathcal{A}).
  \end{align*}
\end{proof}

\begin{remark}
Due to Lemma \ref{lem:eigvals_truncated} above, an optimal preconditioner $\hat{\mathcal{B}}_{\tt bd}$ for 
the truncated system $\hat{\mathcal{A}}$
is given as follows
\begin{align}
  \hat{\mathcal{B}}_{\tt bd} =
  \begin{pmatrix}
    \hat{K} + \eta^{-1/2} \hat{M} & \\
    & \eta \bar{K} + \eta^{1/2} \hat{M}
  \end{pmatrix},
\end{align}
where $\hat{M} = TM,$ where $T$ is defined in \eqref{eqn:TandThat}, and
$\bar{K}$ and $\hat{K}$ are defined in \eqref{eqn:not}. 
\end{remark}

For comparison, we consider block triangular preconditioners of the form
used in Bosch et. al. \cite{Bosch2014}. In the following, we briefly describe
this preconditioner in our notation.
\subsection{\label{sec:precon3}Block Tridiagonal Schur
  Complement Preconditioner (BTDSC)}
In Bosch et. al. \cite{Bosch2014}, a preconditioner is proposed in the framework
of a semi-smooth Newton method combined with Moreau-Yosida regularization for
the same problem. However, the preconditioner was constructed for a linear
system which is different from the one we consider here in  \eqref{eqn:simplified}.  The
preconditioner proposed in \cite{Bosch2014} has the following block lower
triangular form
\begin{align} \label{eqn:block_lower} \mathcal{P}_{\tt btdsc} =
  \begin{pmatrix}
    \bar{K}                      & 0                     \\
    M & -S
  \end{pmatrix},
\end{align}
where $S = \eta \bar{K} + M \bar{K}^{-1}M^T$ is the negative Schur complement.
From Lemma \ref{spd}, $\bar{K}$ is SPD, hence, invertible and from Remark \ref{schur_spd} $S$ is also invertible. Hence by block $2 \times 2$ inversion formula, we have
\begin{align}
  \mathcal{P}_{\tt btdsc}^{-1} =
  \begin{pmatrix}
    \bar{K}                      & 0                     \\
    M & -S
  \end{pmatrix}^{-1} =
  \begin{pmatrix}
    \bar{K}^{-1}                 & 0                     \\
    S^{-1}M^T \bar{K}^{-1} & -S^{-1}
  \end{pmatrix}.
\end{align}
Let $S_{\tt pre}$ be an approximation of Schur complement $S$ in
$\mathcal{P}_{\tt btdsc}$ in
\eqref{eqn:block_lower}, then the new preconditioner $\mathcal{B}_{\tt btdsc},$ and
the corresponding preconditioned operator $\mathcal{B}_{\tt btdsc}^{-1}\mathcal{A}$ are
given as follows
\begin{align}
  \mathcal{B}_{\tt btdsc} = \begin{pmatrix}
    \bar{K}                      & 0                     \\
    M & -S_{\tt pre}
  \end{pmatrix}, \quad \mathcal{B}_{\tt btdsc}^{-1}\mathcal{A} = \begin{pmatrix}
    I                            & \bar{K}^{-1}M^T \\
    0 & S_{\tt pre}^{-1}S
  \end{pmatrix}. \label{eqn:BinvA}
\end{align}
In this paper, we choose a preconditioner $S_{\tt pre}$ for $S$ as follows
\begin{align}
  \label{eqn:schur_complement}
  S_{\tt pre} = (M + \sqrt \eta \bar{K})\bar{K}^{-1}(M +
  \sqrt \eta \bar{K}) = (\eta \bar{K} + M \bar{K}^{-1}M) + 2 \sqrt{\eta}M = S +
  2 \sqrt{\eta}M.
\end{align}     
Such preconditioners had been used before for example, in \cite{Bosch2014,Pearson2012}. We
note the following trivial result.
\begin{lemma} \label{lem:Stilde}
$S_{\tt pre}$ is SPD.
\end{lemma}
\begin{proof}
Follows from \eqref{eqn:schur_complement} and from Theorem \ref{schur_spd} that $M$ and $S$ are SPD.
\end{proof}
In view of \eqref{eqn:BinvA}, the following fact follows.
\begin{fact} \label{thm:eigleft} Let $\mathcal{B}_{\tt btdsc}$ be defined as in
  \eqref{eqn:BinvA}, then there are $|\mathcal{N}_h|$ eigenvalues of
  $\mathcal{B}_{\tt btdsc}^{-1}\mathcal{A}$ equal to one, and the rest are the
  eigenvalues of the preconditioned Schur complement $S_{\tt pre}^{-1}S.$
\end{fact}
In view of Fact \ref{thm:eigleft}, it is sufficient to estimate eigenvalues of
the preconditioned Schur complement.  Using \eqref{eqn:schur_complement} and the
fact that both $S_{\tt pre}$ and $S$ are SPD from Lemma \ref{lem:Stilde} and from Lemma
\ref{schur_spd} respectively, looking at the Rayleigh quotient with $v^Tv=1, v
\in \mathbb{R}^{|\mathcal{N}_h|},$ and using the fact that $\bar{K}$ and $M$ are SPD,
consequently, $\eta \bar{K} + M \bar{K}^{-1}M$ is SPD, we have
\begin{align}
  \frac{v^T (S) v}{v^T (S_{\tt pre}) v} &= \frac{v^T (\eta \bar{K} + M
    \bar{K}^{-1} M)v}{v^T (\eta \bar{K} + M \bar{K}^{-1} M)v + 2\sqrt{\eta}\,
    v^TMv} = \frac{1}{1 + Z},
\end{align}
where 
\begin{align}
   Z &=  \frac{2 v^T \sqrt{\eta}M v}{ v^T (\eta \bar{K} + M \bar{K}^{-1} M)v }. 
  \end{align}
   We have
  \begin{align}
    \min_{v} \: Z &= \min_v \: \frac{2 \cdot \sqrt{\eta} \cdot v^T M v}{ v^T
      (\eta \bar{K} + M
      \bar{K}^{-1} M)v } \\
    &= \min_{v} \: \frac{2 \sqrt{\eta}}{\eta v^T M^{-1}\bar{K} v + v^T \bar{K}^{-1}Mv } \\
    &= \min_{v} \: \frac{\bar{K}^{1/2}}{\bar{K}^{1/2}} \cdot \frac{2
      \sqrt{\eta}}{\eta v^T M^{-1}\bar{K} v +
      v^T \bar{K}^{-1}Mv} \cdot \frac{\bar{K}^{-1/2}}{\bar{K}^{-1/2}}  \\
    &= \min_{v} \: \frac{2 \sqrt{\eta}}{\eta v^T \bar{K}^{1/2}M^{-1}\bar{K}^{1/2}v + v^T \bar{K}^{-1/2} M \bar{K}^{-1/2}M v } \\
    &= \min_{v} \: \frac{2 u^T w}{u^Tu + w^T w},
\end{align}
where $u = \sqrt{\eta}M^{-1/2}\bar{K}^{1/2}v$ and $w = M^{1/2} \bar{K}^{-1/2}v.$
Similarly, 
 \begin{align*}
 \max_{v} \: Z &=  \max_{v} \: \frac{2 u^T w}{u^Tu + w^T w}.
 \end{align*}
Since $\eta > 0,$ $(u-w)^T(u-w) \geq 0,$ and that $u^Tu + w^Tw
> 0,$ we clearly have
\begin{align}
  0 < \frac{2 u^T w}{u^Tu + w^T w} \leq 1,
\end{align}
which leads to the following bounds
\begin{align}
  \label{eqn:SandSpre}
  \frac{1}{2} \leq \lambda_{\min}(S_{\tt pre}^{-1}S) = \min_{v \neq 0}
  \frac{v^T(-S)v}{v^T (-S_{\tt pre})v} 
  \leq \max_{v \neq 0} \frac{v^T(-S)v}{v^T (-S_{\tt pre})v} =
  \lambda_{\max}(S_{\tt pre}^{-1}S) < 1.
\end{align}
We note this result as theorem below.
\begin{theorem}
  \label{eig_estimate_bosh}
  The eigenvalues of the preconditioned untruncated system $S_{\tt pre}^{-1}S$
  satisfies
\begin{align}  
  \lambda(S_{\tt pre}^{-1}S) \in [1/2, 1).
  \end{align}
\end{theorem}
\begin{corollary}
The condition number is bounded as follows
\begin{align}
 \kappa(S_{\tt pre}^{-1}S) < 2.
\end{align}
\end{corollary}
\begin{remark}
  When using GMRES \cite{Saad2003}, right preconditioning is preferred. As in
  Theorem \ref{thm:eigleft}, similar estimate for the right preconditioned
  matrix $S S_{\tt pre}^{-1}$ holds, because both $S_{\tt pre}^{-1}S$ and
  $S^{-1}S_{\tt pre}$ are similar to a symmetric matrix $S_{\tt pre}^{-1/2} S
  S_{\tt pre}^{-1/2}.$
\end{remark}
Let $x=[x_1, x_2], b = [b_1, b_2].$ The preconditioned system
$\mathcal{B}_{\tt btdsc}^{-1}\mathcal{A}x = \mathcal{B}_{\tt
  btdsc}^{-1}b$ is given as follows
\begin{align}
\begin{pmatrix}
    I                            & \bar{K}^{-1}M^T \\
    0                            & S_{\tt pre}^{-1}S
  \end{pmatrix}
  \begin{pmatrix}
    x_1                                                  \\ x_2
  \end{pmatrix} =
    \begin{pmatrix}
    \bar{K}^{-1}                 & 0                     \\
    S_{\tt pre}^{-1}M^T \bar{K}^{-1} & -S_{\tt pre}^{-1}
  \end{pmatrix}
  \begin{pmatrix}
    b_1                                                  \\ b_2
  \end{pmatrix}
\end{align}
from which we obtain the following set of equations
\begin{align}
  x_1 + \bar{K}^{-1}M^T x_2 = \bar{K}^{-1}b_1, \quad S_{\tt pre}^{-1}Sx_2 =
  S_{\tt pre}^{-1}(M^T \bar{K}^{-1} b_1 - b_2).
\end{align}
\begin{alg}\label{alg:solve}
Objective: Solve $\mathcal{B}_{\tt btdsc}^{-1}\mathcal{A}x = \mathcal{B}^{-1}b$
  \begin{enumerate}
  \item Solve for $x_2:$ $S_{\tt pre}^{-1}Sx_2 =
    S_{\tt pre}^{-1}(M^T\bar{K}^{-1}b_1 - b_2)$
  \item Set $x_1 = \bar{K}^{-1}(b_1 - M^Tx_2)$
  \end{enumerate}
\end{alg}
Here if Krylov subspace method is used to solve for $x_2$, then matrix vector
product with $S$ and a solve with $S_{\tt pre}$ is needed. However, when the
problem size, i.e., $|\mathcal{N}_h|$ is large, it won't be feasible to do exact
solve with $\bar{K},$ and we need to solve it inexactly, for example, using
algebraic multigrid methods. In the later case, the decoupling of $x_1$ and
$x_2$ as in Algorithm \ref{alg:solve} is not possible; then we use GMRES
\cite[p. 269]{Saad2003} preconditioned by $\mathcal{B}_{\tt btdsc}.$

In view of Fact \ref{thm:eigleft} and Theorem \ref{eig_estimate_bosh}, we
already have eigenvalue estimates for $\mathcal{B}^{-1}_{\tt btdsc}\mathcal{A},$
however, as before, we can derive the eigenvalue bound and condition number
estimate for $\mathcal{B}^{-1}_{\tt btdsc}\mathcal{A}$ directly without
explicitely reducing it to Schur complement system. To this end, we consider
again the related generalized eigenvalue problem
\begin{align}
  \label{eq:gen_eig_bdsc}
  \begin{pmatrix}
    \bar{K} & M \\
    M & - \eta \bar{K}
  \end{pmatrix}\begin{pmatrix}
    u \\ v
  \end{pmatrix}
  &= \lambda
  \begin{pmatrix}
    \bar{K} & \\
     M  & -\eta (\bar{K} + \eta^{-1/2} M) \bar{K}^{-1} (\bar{K} +
      \eta^{-1/2} M) 
 \end{pmatrix}  \begin{pmatrix}
    u \\ v
  \end{pmatrix}.  
\end{align}
Note that we have rewritten $S_{\tt pre}$ in  \eqref{eq:gen_eig_bdsc} as follows
\begin{align}
S_{\tt pre} = -\eta (\bar{K} + \eta^{-1/2} M) \bar{K}^{-1} (\bar{K} +
      \eta^{-1/2} M).
\end{align}
From \eqref{eq:gen_eig_bdsc}, we have 
\begin{align}
  \label{eq:btdsc_eig_eqns}
  \bar{K}u + M v &= \lambda \bar{K}u \\
  Mu - \eta \bar{K}v &= \lambda (Mu - \eta (\bar{K} + \eta^{-1/2}M) \bar{K}^{-1}
  (\bar{K} + \eta^{-1/2}M)v ).
\end{align}
As before, we consider the eigenvalue problem \eqref{eqn:gen_eig} with the
eigenbasis $\{e_1, e_2, \dots, \}$ which are orthonormal
w.r.t. $\bar{K}+\eta^{-1/2}M$ inner product.
Expanding $u$ and $v$ in eigenbasis $\{ e_1, e_2, \dots, \}$ as in \eqref{eqn:uAndv},
and looking at the $i$th
rows of these two equations, we get 
\begin{align}
  \label{eq:expanded}
  \hat{u}_i \bar{K}e_i + \hat{v}_iMe_i &= \lambda \hat{u}_i \bar{K}e_i \\
  \hat{u}_iMe_i - \eta \hat{v}_i\bar{K}e_i &= \lambda (\hat{u}_i Me_i - \eta
  \hat{v}_i (\bar{K} + \eta^{-1/2}M) \bar{K}^{-1} (\bar{K} + \eta^{-1/2}M)e_i).
\end{align}
Again from \eqref{eqn:gen_eig}
\begin{align}
  \label{eq:ith_row_Mei_2}
  Me_i = \eta^{1/2}(1 - \mu_i) \mu_i^{-1} \bar{K}e_i.
\end{align}  
Substituting $Me_i$ from above in two equations of \eqref{eq:expanded}, we have
\begin{align}
  \label{eq:eig_eqns_2_1}
  \hat{u}_i \bar{K}e_i + \hat{v}_i \eta^{1/2} (1-\mu_i)\mu_i^{-1} \bar{K}e_i &=
  \lambda \hat{u}_i \bar{K}e_i \\
  \hat{u}_i \eta^{1/2} (1- \mu_i) \mu_i^{-1} \bar{K}e_i - \eta \hat{v}_i
  \bar{K}e_i &= \lambda \hat{u}_i \eta^{1/2} (1 - \mu_i) \mu_i^{-1} \bar{K}e_i -
  \lambda \eta {\mu_i}^{-2} \bar{K} e_i \hat{v}_i. \label{eq:eig_eqns_2_2}
\end{align}   
Multiplying by $e_i^T$ from the left and dividing by $e_i^T \bar{K} e_i \neq 0$
throughout, we have
\begin{align}
  \label{eq:ith_row3}
  \hat{u}_i + \eta^{1/2}(1-\mu_i)\mu_i^{-1} \hat{v}_i &= \lambda \hat{u}_i \\
  \eta^{1/2}(1 - \mu_i)\mu_i^{-1} \hat{u}_i - \eta \hat{v}_i &= 
  \lambda \eta^{1/2} (1 - \mu_i)\mu_i^{-1}\hat{u}_i - \lambda \eta \mu_i^{-2} \hat{v}_i. 
\end{align}
Rearranging above,
\begin{align}
  \label{eq:ith_eig}
  \begin{pmatrix}
    1 & \eta^{1/2} (1 - \mu_i)\mu_i^{-1} \\
    \eta^{1/2} (1 - \mu_i)\mu_i^{-1}(1 - \lambda) & -\eta 
  \end{pmatrix}
  \begin{pmatrix}
    \hat{u}_i \\ \hat{v}_i
  \end{pmatrix} =
  \lambda \begin{pmatrix}
    1 & \\
     &  -\eta \mu_i^{-2}
  \end{pmatrix}\begin{pmatrix}
    \hat{u}_i \\ \hat{v}_i
  \end{pmatrix}.
\end{align}
There exists at least one $i$ such that 
\begin{align}
  \label{eq:det}
  \text{det}\begin{pmatrix}
    1 - \lambda & \eta^{1/2} (1 - \mu_i)\mu_i^{-1} \\
    \eta^{1/2} (1 - \mu_i)\mu_i^{-1}(1 - \lambda) & -\eta (1 - \lambda \mu_i^{-2})
  \end{pmatrix}=0,
\end{align}
which implies
\begin{align}
  \label{eq:det2}
  (1 - \lambda \mu_i^{-2}) + (1 - \mu_i)^2 \mu_i^{-2} &= 0 \\
  \implies \lambda = \mu_i^2 + (1 - \mu_i)^2.
\end{align}
The function $f(\mu_i) = \mu_i^2 + (1 - \mu_i)^2$ has a critical point at $\mu_i
= 1/2,$ and $f(\cdot)$ monotonically decreases from 1 to 1/2 for $\mu_i \in
(0,1/2],$ and monotonically increases from 1/2 to 1 for $\mu_i \in [1/2, 1].$
All this leads to the following bound.
\begin{theorem}
  \label{eig_bound2} There holds
\begin{align}
  \lambda(\mathcal{B}_{\tt btdsc}^{-1} \mathcal{A}) \in [1/2, 1).
\end{align}
\end{theorem}
\begin{corollary}
    The condition number satisfies the following bound 
  \begin{align}
  \kappa(\mathcal{B}_{\tt btdsc}^{-1} \mathcal{A}) &<2.
  \end{align}
\end{corollary}

\begin{remark}[{\bf Relation between eigenvalues of truncated and untruncated
    system}]
\label{rem:relation_eig_vals}
  We have two cases
  \begin{enumerate}
  \item $(1,1)$ block is solved inexactly: as mentioned before, in this case, the preconditioner is
    block tridiagonal hence unsymmetric, consequently, Fischer theorem cannot be
    used to show relation between truncated and untruncated system
  \item $(1,1)$ block is solved exactly: in this case, the problem 
    reduces to Schur complement system, and due to Lemma \eqref{schur_spd}, the
    truncated Schur complement remains SPD. The preconditioner for truncated Schur complement $\hat{S}$ is defined below
    \begin{align}
      {\hat{S}}_{\tt pre} = (\hat{M} + \sqrt{\eta} \bar{K})
      \hat{K}^{-1}(\hat{M} + \sqrt{\eta} \hat{K}) &= \eta \bar{K} +
      \hat{M} \hat{K}^{-1} \hat{M}^T + \sqrt{\eta} \hat{M} +
      \sqrt{\eta}\bar{K} \hat{K}^{-1}\hat{M}^T \label{spre} \\ 
      &= \hat{S} + \sqrt{\eta} (\hat{M} + \bar{K} \hat{K}^{-1}\hat{M}^T).
    \end{align}
    First, it is not evident whether ${\hat{S}}_{\tt pre}$ is similar to a symmetric
    matrix. If it is, then we want to know whether the following holds    
    \begin{align}
      \lambda_{\max}({\hat{S}}_{\tt pre}^{-1}\hat{S}) \leq
      \lambda_{\max}({S}_{\tt pre}^{-1}S), \quad \lambda_{\min} ({\hat{S}}_{\tt pre}^{-1}\hat{S}) \geq
      \lambda_{\min}({S}_{\tt pre}^{-1}S). \label{eig_precon_sch}
    \end{align}
    We leave this as a subject of future work. Since $S_{\tt pre}$ may be unsymmetric, we shall use \eqref{spre} with GMRES that allows unsymmetric preconditioners.
  \end{enumerate}
\end{remark}
\subsection{\label{sec:precon2}Block Diagonal Schur Complement Preconditioner (BDSC)}
Substituting $\theta=0$ in \eqref{eqn:templ}, we obtain a block diagonal
preconditioner involving Schur complement as follows
\begin{align}
  \mathcal{B}_{\tt bdsc} =
  \begin{pmatrix}
    2\bar{K} & \\
    & S
  \end{pmatrix} \sim \begin{pmatrix}
    \bar{K} & \\
    & S
  \end{pmatrix},
\end{align}
where $S = \eta \bar{K} + M \bar{K}^{-1} M.$ Once again $S$ is approximated by
$S_{\tt pre}$ as before.
\begin{remark}
  As in Lemma \ref{lem:eigvals_truncated}, we have
  \begin{align}
    \lambda_{\min} (\widecheck{\mathcal{B}}_{\tt
      bdsc}^{-1}\widecheck{\mathcal{A}}) \geq \lambda_{\min}(\mathcal{B}_{\tt
      bdsc}^{-1}\mathcal{A}), \quad \lambda_{\max} (\widecheck{\mathcal{B}}_{\tt
      bdsc}^{-1}\widecheck{\mathcal{A}}) \geq \lambda_{\max}(\mathcal{B}_{\tt
      bdsc}^{-1}\mathcal{A}),
  \end{align}
  which suggests the following optimal preconditioner for
  $\hat{\mathcal{A}}$
  \begin{align}
    \hat{\mathcal{B}}_{\tt bdsc} =
    \begin{pmatrix}
      \hat{K} & \\
       & S 
    \end{pmatrix},
  \end{align}
  moreover, due the spectral equivalence of $S_{\tt pre}$ and $S$ established in
  \eqref{eqn:SandSpre}, we propose the following preconditioner using same notation
  \begin{align} 
    \hat{\mathcal{B}}_{\tt bdsc} =
    \begin{pmatrix}
      \hat{K} & \\
       & S_{\tt pre}
    \end{pmatrix}.
  \end{align}
  In practice, we shall replace
  $S_{\tt pre}$ by $\hat{S}_{\tt pre}$ defined in \eqref{spre}. 
\end{remark}
As before, we consider the eigenvalue problem \eqref{eqn:gen_eig} with the
eigenbasis $\{e_1, \dots, \}$ which are orthonormal
w.r.t. $\bar{K}+\eta^{-1/2}M$ inner product.  Consider the following generalized
eigenvalue problem
\begin{align}
  \label{eq:gen_eig_bdsc2}
  \begin{pmatrix}
    \bar{K} & M \\
    M & - \eta \bar{K}
  \end{pmatrix}\begin{pmatrix}
    u \\ v
  \end{pmatrix}
  &= \lambda
  \begin{pmatrix}
    \bar{K} & \\
      & \eta (\bar{K} + \eta^{-1/2} M) \bar{K}^{-1} (\bar{K} +
      \eta^{-1/2} M) 
 \end{pmatrix}  \begin{pmatrix}
    u \\ v
  \end{pmatrix}  
\end{align}
which leads to 
\begin{align}
  \label{eq:eig_eqns}
  \bar{K}u + Mv &= \lambda \bar{K} u \\
  Mu - \eta \bar{K} v &= \lambda \eta (\bar{K} + \eta^{-1/2} M) \bar{K}^{-1}
  (\bar{K} + \eta^{-1/2} M)v.
\end{align}
As before, expanding $u$ and $v$ in eigenbasis $\{ e_1, e_2, \dots, \}:$ 
\begin{align}
  \label{eq:eig_basis_u_v}
  u = \sum_i \hat{u}_i e_i, \quad v = \sum_i \hat{v}_i e_i,
\end{align}
and substituting $u$ and $v$ from above in \eqref{eq:eig_eqns}, and looking at
the $i$th rows of both equations, we have
\begin{align}
  \label{eq:ith_rows}
  \hat{u}_i \bar{K} e_i + \hat{v}_i M e_i &= \lambda \hat{u}_i \bar{K}e_i \\
  \hat{u}_i Me_i - \eta \hat{v}_i \bar{K}e_i &= \lambda \eta (\bar{K} +
  \eta^{-1/2} M) \bar{K}^{-1}\hat{v}_i (\bar{K} +
  \eta^{-1/2} M)e_i.
\end{align}
Again from \eqref{eqn:gen_eig}
\begin{align}
  \label{eq:ith_row_Mei}
  Me_i = \eta^{1/2}(1 - \mu_i) \mu_i^{-1} \bar{K}e_i.
\end{align}
Substituting $Me_i$ from above in \eqref{eq:ith_rows}, we have
\begin{align}
  \label{eq:eig_eqns_2_1}
  \hat{u}_i \bar{K}e_i + \hat{v}_i \eta^{1/2} (1-\mu_i)\mu_i^{-1} \bar{K}e_i &=
  \lambda \hat{u}_i \bar{K}e_i \\
  \hat{u}_i \eta^{1/2} (1- \mu_i) \mu_i^{-1} \bar{K}e_i - \eta \hat{v}_i
  \bar{K}e_i &= \lambda \eta (\bar{K} + \eta^{-1/2}M) \bar{K}^{-1} \hat{v}_i
  (\bar{K}e_i + \eta^{-1/2} \eta^{1/2} (1 - \mu_i) \mu_i^{-1} \bar{K} e_i) \\
  &= \lambda \eta (\bar{K} + \eta^{-1/2}M)  (e_i + (1 - \mu_i)
  \mu_i^{-1} e_i) \hat{v}_i\\
  &=  \lambda \eta (\bar{K} + \eta^{-1/2}M) e_i (1 + (1-\mu_i)
  \mu_i^{-1}) \hat{v}_i\\
  &= \lambda \eta  (1 + (1 - \mu_i) \mu_i^{-1}) \mu_i^{-1} \bar{K} e_i \hat{v}_i, \quad (\text{from}~ \eqref{eqn:gen_eig}).
  \\ 
  &= \lambda \eta \mu^{-2} \bar{K} e_i \hat{v}_i. \label{eq:eig_eqns_2_2}
\end{align}   
Multiplying \eqref{eq:eig_eqns_2_1} and \eqref{eq:eig_eqns_2_2} from left by $e_i^T$
and multiplying \eqref{eq:eig_eqns_2_2} by $\mu_i^2,$ and cancelling $e_i^T \bar{K} e_i$ from both equations, we obtain
\begin{align}
  \label{eq:ith_row_red}
  \hat{u}_i + \hat{v}_i \eta^{1/2} (1-\mu_i) \mu_i^{-1} &= \lambda \hat{u}_i \\
  \hat{u}_i \eta^{1/2} (1 - \mu_i) \mu_i - \eta \mu_i^2 \hat{v}_i &= \lambda
  \eta \hat{v}_i,
\end{align}
writing in matrix form, we obtain
\begin{align}
  \label{eq:ith_row_matrix}
  \begin{pmatrix}
    1 & \eta^{1/2} (1 - \mu_i) \mu_i^{-1} \\
    \eta^{1/2} (1 - \mu_i) \mu_i & -\eta \mu_i^2
  \end{pmatrix} = \lambda
  \begin{pmatrix}
    1 &  \\
     & \eta
  \end{pmatrix}
  \begin{pmatrix}
    \hat{u}_i \\ \hat{v}_i
  \end{pmatrix}.
\end{align}
Since $u$ and $v$ are eigenvectors, there exists at least one $i$ such that
following holds
\begin{align}
  \label{eq:det_i}
  \text{det} \left(
    \begin{pmatrix}
      1 & \eta^{1/2}(1 - \mu_i)\mu_i^{-1} \\
      \eta^{1/2} (1 - \mu_i)\mu_i & -\eta \mu_i^2
    \end{pmatrix} - \lambda
    \begin{pmatrix}
      1 & \\
        & \eta
    \end{pmatrix}
\right) &= 0 \\
\text{det}\left(
  \begin{pmatrix}
    1 - \lambda & \eta^{1/2} (1 - \mu_i) \mu_i^{-1} \\
    \eta^{1/2} (1 - \mu_i) \mu_i & -\eta (\mu_i^2 + \lambda)
  \end{pmatrix} 
\right) &= 0 \\
\implies -\eta (1 - \lambda) (\mu_i^2 + \lambda) - \eta (1 - \mu_i)^2 &= 0 \\
\implies (1 - \lambda) (\mu_i^2 + \lambda) + (1 - \mu_i)^2 &= 0 \\
\implies \mu_i^2 + \lambda - \lambda \mu_i^2 - \lambda^2 + 1 + \mu_i^2 - 2 \mu_i
&= 0 \\ \label{eqn:quadratic_eqn}
\implies - \lambda^2 + \lambda (1 - \mu_i^2) + 2 \mu_i^2 - 2 \mu_i + 1 &=
0. 
\end{align}  
The equation \eqref{eqn:quadratic_eqn} has two roots as follows
\begin{align}
  \label{eqn:roots} 
  \lambda_1(\mu_i) = \frac{1 - \mu_i^2}{2} + \frac{\mu_i^4 + 6\mu_i^2 - 8\mu_i + 5}{2},
  \quad \lambda_2(\mu_i) = \frac{1 - \mu_i^2}{2} - \frac{\mu_i^4 + 6\mu_i^2 - 8\mu_i + 5}{2},
\end{align} 
with the constraints that $\mu_i \in (0,1].$ The critical points of the
first equation in \eqref{eqn:roots} is given by the roots of the following
equation
\begin{align}
  \label{eq:critical_roots}
  \frac{d \lambda_1}{d\mu_i} = (4\mu_i^3 + 12\mu_i - 8)/4(\mu_i^4 + 6\mu_i^2 - 8\mu_i + 5)^{1/2} - \mu_i = 0.
\end{align}
The roots are $\mu_i= 1, -\sqrt{2}-1,$ where the last one is discarded since it
is outside the constraint interval $(0,1].$ Since only the boundary points are
critical points, $\lambda_1$ is either monotonically increasing or monotonically
decreasing, but by checking, we have $\lambda_1(0) = (\sqrt{5}+1)/2, \: \lambda_1(1) =
1,$ thus, $\lambda_1$ is monotonically decreasing for $\mu_i \in (0,1].$ Thus
$\lambda_1 \in [1, (\sqrt{5}+1)/2].$ Similarly, we now consider the second root
 $\lambda_2$ in \eqref{eqn:roots} whose critical points are given by the roots
of
\begin{align}
  \label{eq:critical_second}
 \frac{d \lambda_2}{d\mu_i}  = - \mu_i - (4\mu_i^3 + 12\mu_i - 8)/(4(\mu_i^4 + 6\mu_i^2 - 8\mu_i +
  5)^{1/2}) = 0,
\end{align}
and it has repeated roots $\mu_i = \sqrt{2}-1.$ To determine whether it is a
maxima or minima, we consider
\begin{align}
  \label{eq:double_derivative}
  \frac{d^2 \lambda_2}{d\mu_i^2} = (4{\mu_i}^3 + 12\mu_i - 8)^2/(8({\mu_i}^4 + 6{\mu_i}^2 - 8\mu_i + 5)^{3/2}) - (12{\mu_i}^2 +
  12)/(4({\mu_i}^4 + 6{\mu_i}^2 - 8\mu_i + 5)^{1/2}) - 1,
\end{align}
which is negative for $\mu_i = \sqrt{2}-1,$ thus, it is a maxima for which
$\lambda_2$ attains the value $1 - \sqrt{2}.$ Since there are no other critical
points, the minima must occur at one of the two boundaries of $(0, 1].$ For
$\mu_i=0, \lambda_2 = (1 - \sqrt{5})/2,$ and for $\mu_i=1, \lambda_2 = -1.$ Thus we
have the following bound for eigenvalues.
\begin{theorem}[{\bf Eigenvalue bounds of $\mathcal{B}^{-1}_{\tt bdsc}\mathcal{A}$}]
  \label{eig_thm3}
  There holds
  \begin{align}
    \lambda(\mathcal{B}_{\tt bdsc}^{-1}\mathcal{A}) \in [-1, 1-\sqrt{2}] \cup [1, (\sqrt{5}+1)/2].  
  \end{align}
\end{theorem}
The condition number estimate then follows.
\begin{corollary}[{\bf Condition number of $\mathcal{B}^{-1}_{\tt bdsc}\mathcal{A}$}]
  There holds
\begin{align}
  \label{eq:cond3}
  \kappa(\mathcal{B}_{\tt bdsc}^{-1}\mathcal{A}) <
  \frac{\sqrt{5}+1}{2(\sqrt{2}-1)} \approx 3.90.
\end{align}
\end{corollary}

\section{\label{sec:numexp} Numerical Experiments}
All the experiments were performed in double precision arithmetic in MATLAB.  A
fixed number of 12 Uzawa iterations per time step is executed. The obstacle
problem is solved using monotone multigrid.  For the linear subproblem, the
Krylov solver used was restarted GMRES with inner subspace dimension of 60, and
maximum number of iterations allowed was 300. The iteration was stopped as soon
as the relative residual was below the tolerance of $10^{-7}.$ The local
sub-blocks of the preconditioner was solved using aggregation based AMG; the
stopping criteria for AMG was decrease of relative residual below $10^{-7}.$
Three test cases are considered
\begin{itemize}
  \item Evolution of square
  \item Evolution of randomly mixed phases
  \item Randomly truncated systems
\end{itemize}
We describe the numerical experiments with each of these test cases below.  
\subsection{Experiments with Various Evolutions}
In both the evolution problems, we chose $\epsilon=2 \times 10^{-2}$ and
$\tau=10^{-5}.$ We consider the mesh sizes $h=1/256, 1/400$ with 66049 and
160801 nodes respectively. In the Tables \ref{table:evol_rand}, \ref{tab:evol_square}, and \ref{tab:artificial}, we show the number of truncations denoted by {\tt \#trunc}, and
percentage of truncations denoted by {\tt \%trunc} during evolutions.  We recall from \eqref{eqn:trun_simplified}, that we need to solve twice, since, we use Sherman-Woodbury inversion \cite{Golub2013}[(2.1.5), p. 65]: in the tables, the iteration counts for the first solve is denoted by {\tt it1}, and those for the second solve is denoted by {\tt it2}. The {\tt time} in the table denotes the total time in seconds for both these solves. We
compare three preconditioners: {\tt bd}, {\tt bdsc}, and {\tt btdsc}.
\subsubsection{Evolution of Randomly Mixed Phases}
In this test case, we take initial solution $u$ to have random values between
-0.3 and 0.5 except for two pure phases of $u(1)=1$ and $u(\text{end})=-1.$ In Figure
\ref{fig:rand_1}, we show the initial active set configuration. The evolution
for various time steps are shown in Figures \ref{fig:rand_1} to
\ref{fig:rand_200}. For this test case, already at time step $\tau=80,$
about half of the nodes are truncated; suggesting fast separation initially.  The iteration counts for {\tt btdsc} is
the least. Except for {\#trunc=2}, {\tt btdsc} has the least CPU time of all
three preconditioners. Although, {\tt bdsc} has slightly less iterations than
{\tt bd}, the CPU times are large compared to that for {\tt bd}, especially,
initially when the number of truncations are less. The larger CPU times are
attributed to the fact that {\tt bdsc} requires three elliptic solves and one
matrix vector product, whereas, {\tt bd} requires only two elliptic solves.
Being a block tridiagonal preconditioner, {\tt btdsc} has more costs compared to
{\tt bd} and {\tt bdsc}, but since the iteration counts for {\tt btdsc} is
almost half of those for {\tt bd} and {\tt bdsc}, it converges significantly
faster. For this evolution, although truncations increase, the iteration counts
remain steady during various time steps for all three preconditioners. We
observe that initial fast dynamics of phase separation later slows down after
about $\tau=120,$ when we do not see any significant increase in
truncations. This suggests that the system remains structurally and spectrally
similar, this is suggested by the iteration count that remains almost constant
after $\tau=120$ for all three methods.

\begin{figure}[tbp]
\centering  
\subfigure[\label{fig:rand_1}$\tau=1$]{\includegraphics[width=3.4cm, height=3.1cm]{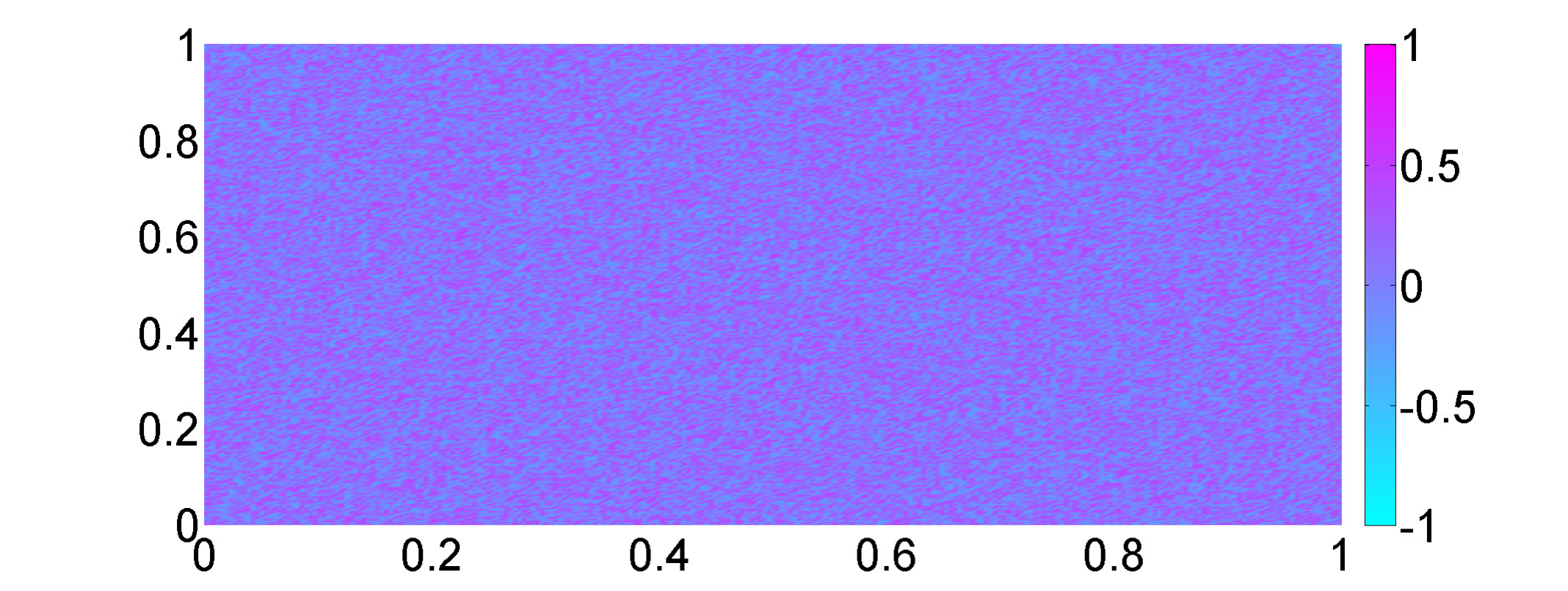}}
\hfill
\centering  
\subfigure[\label{fig:rand_20}$\tau=20$]{\includegraphics[width=3.4cm, height=3.1cm]{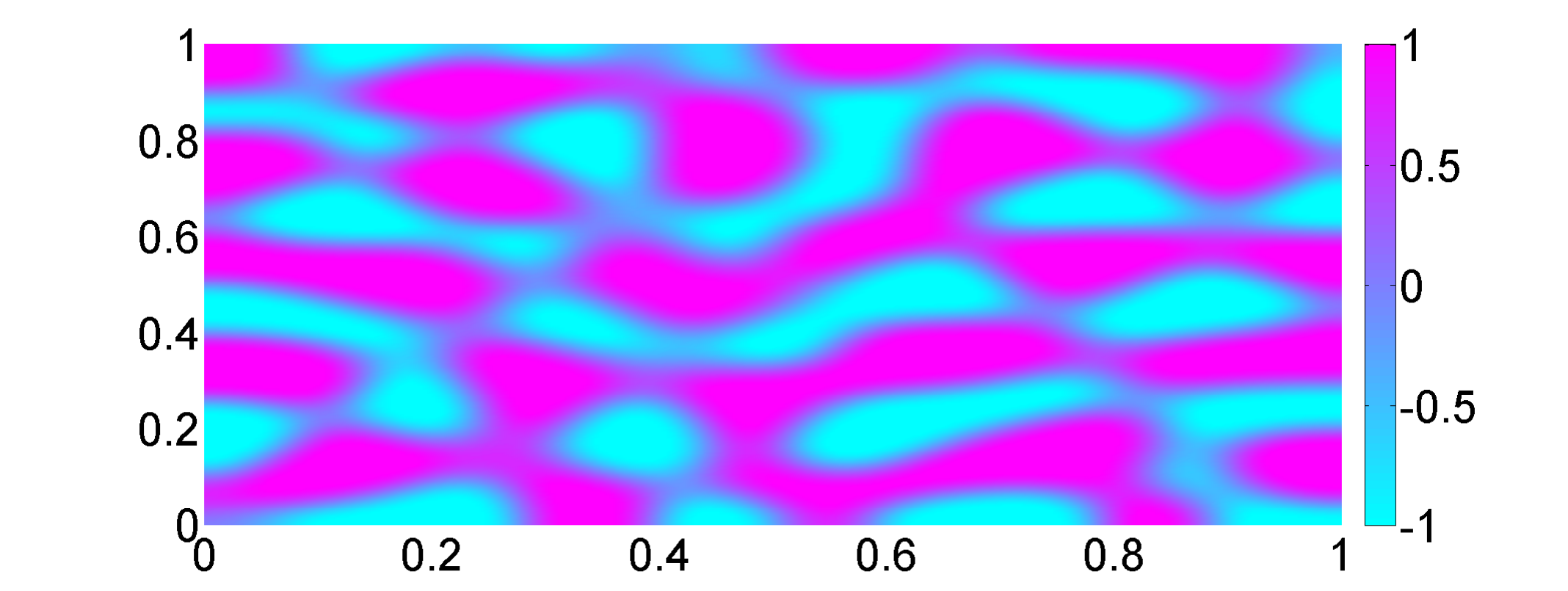}}
\hfill 
\centering  
\subfigure[\label{fig:rand_40}$\tau=40$]{\includegraphics[width=3.4cm, height=3.1cm]{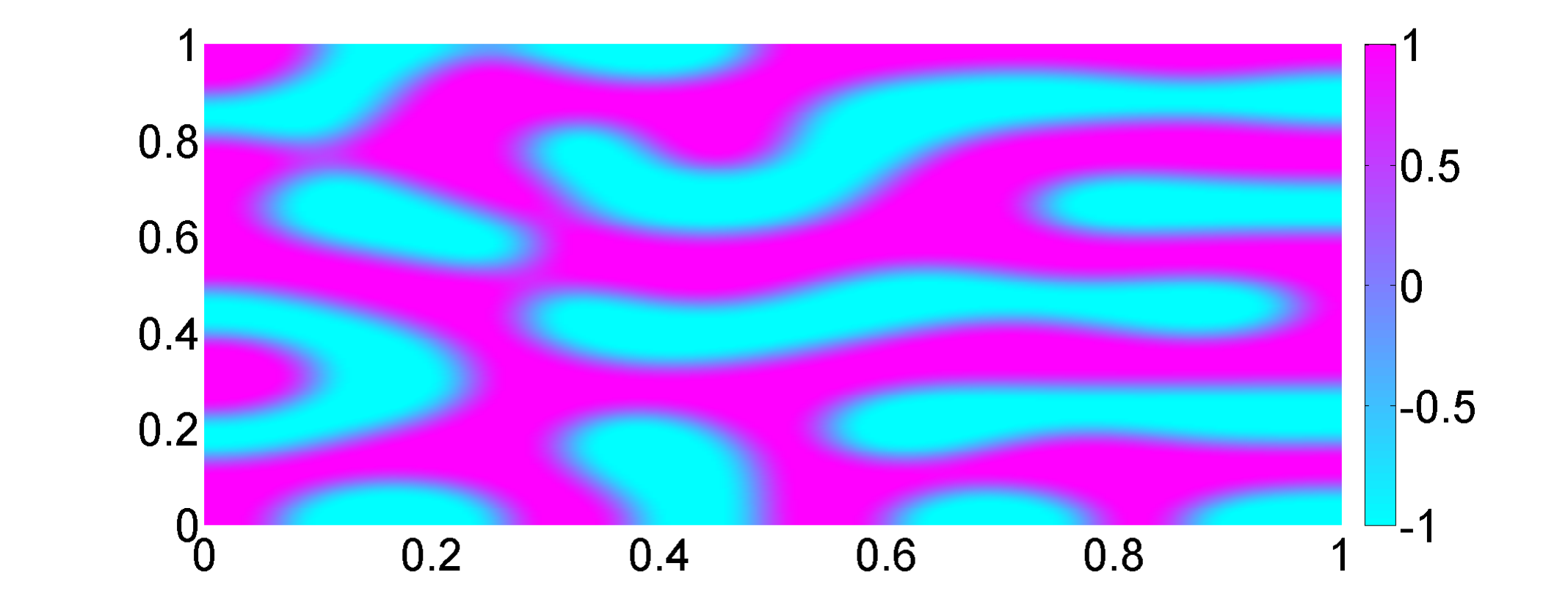}}
\hfill 
\centering  
\subfigure[\label{fig:rand_60}$\tau=60$]{\includegraphics[width=3.4cm, height=3.1cm]{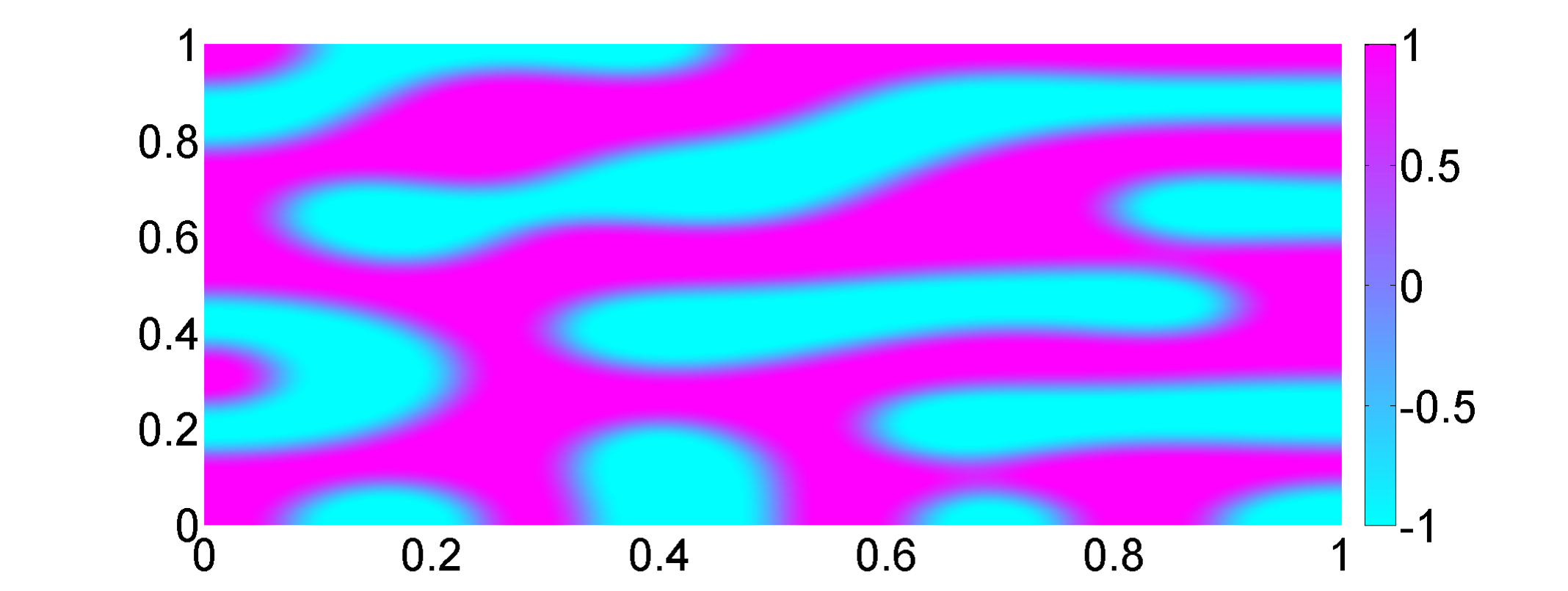}}
\hfill 
\centering  
\subfigure[\label{fig:rand_80}$\tau=80$]{\includegraphics[width=3.4cm, height=3.1cm]{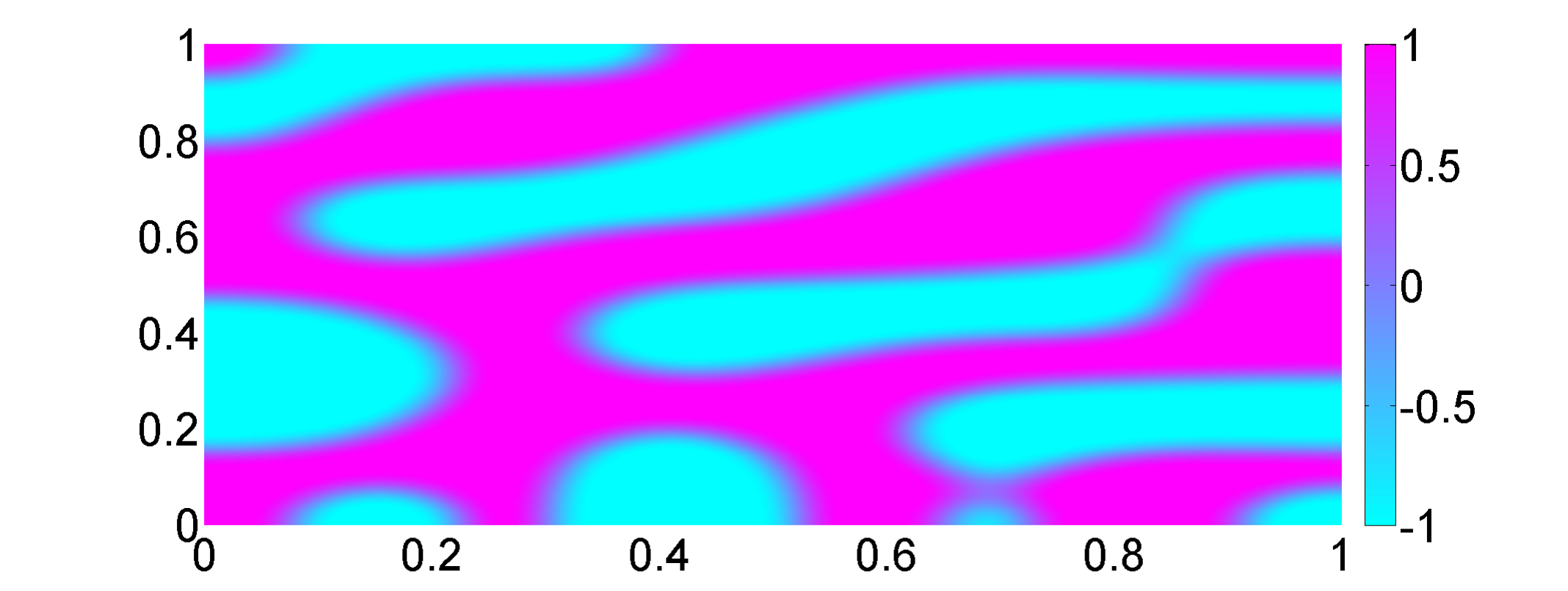}}
\hfill 
\centering  
\subfigure[\label{fig:rand_100}$\tau=100$]{\includegraphics[width=3.4cm, height=3.1cm]{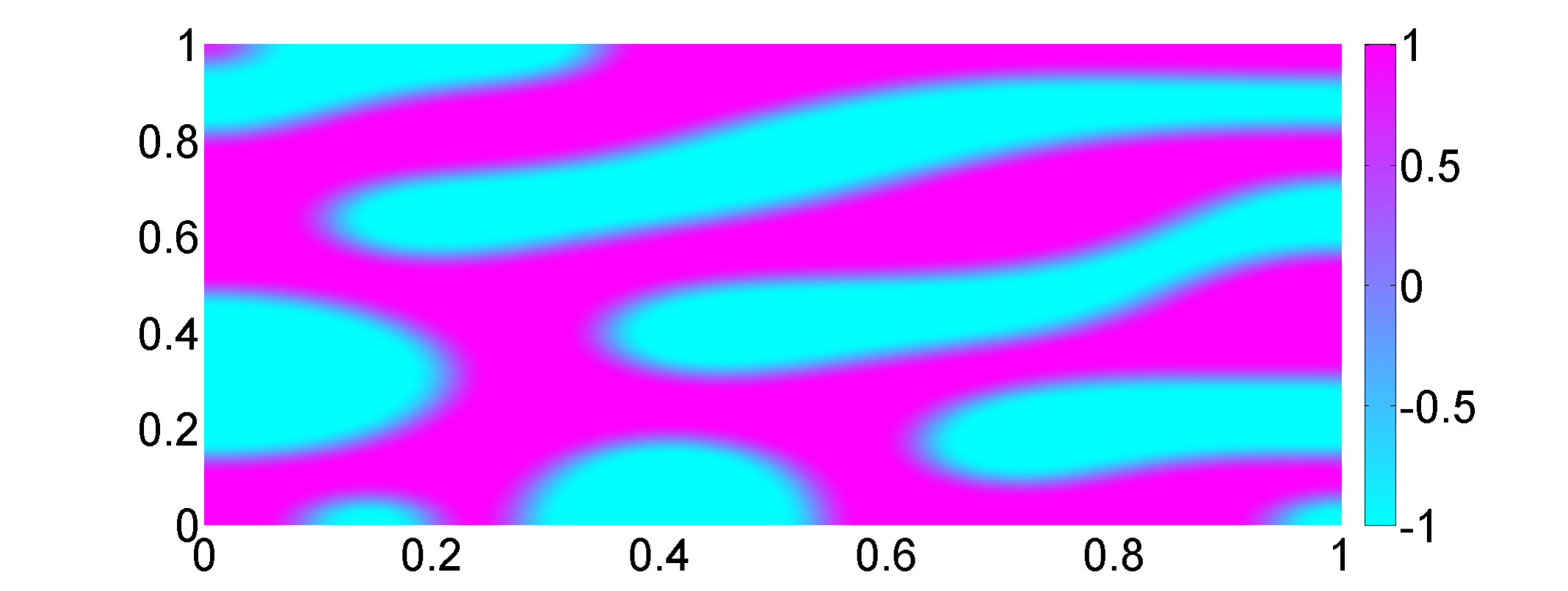}}
\hfill 
\centering  
\subfigure[\label{fig:rand_120}$\tau=120$]{\includegraphics[width=3.4cm, height=3.1cm]{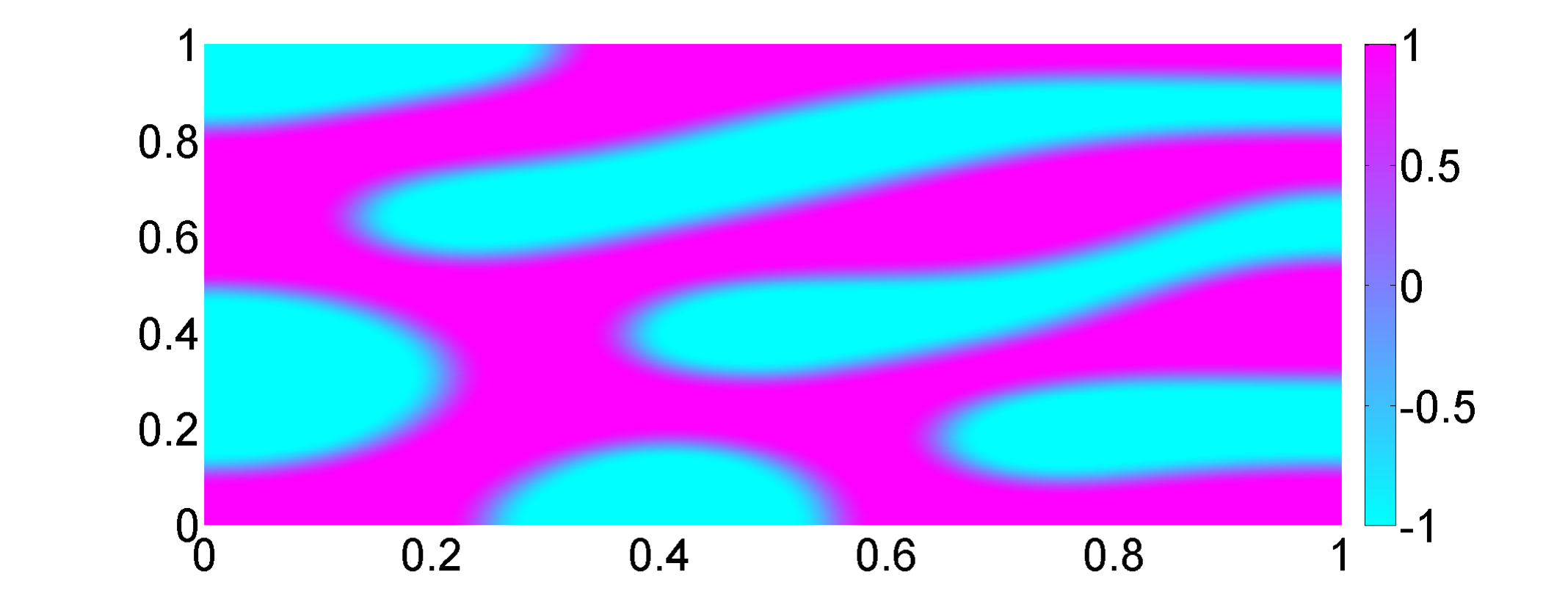}}
\hfill 
\centering  
\subfigure[\label{fig:rand_160}$\tau=160$]{\includegraphics[width=3.4cm, height=3.1cm]{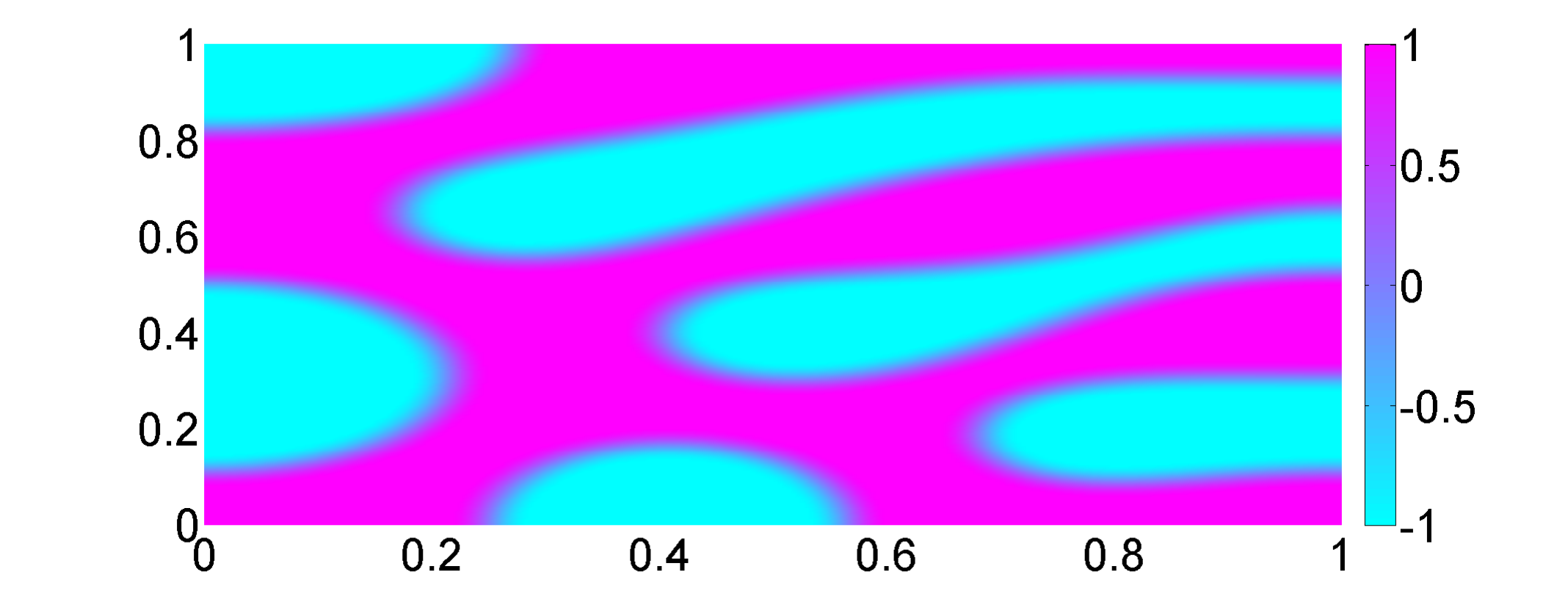}}
\hfill 
\centering 
\subfigure[\label{fig:rand_160}$\tau=180$]{\includegraphics[width=3.4cm, height=3.1cm]{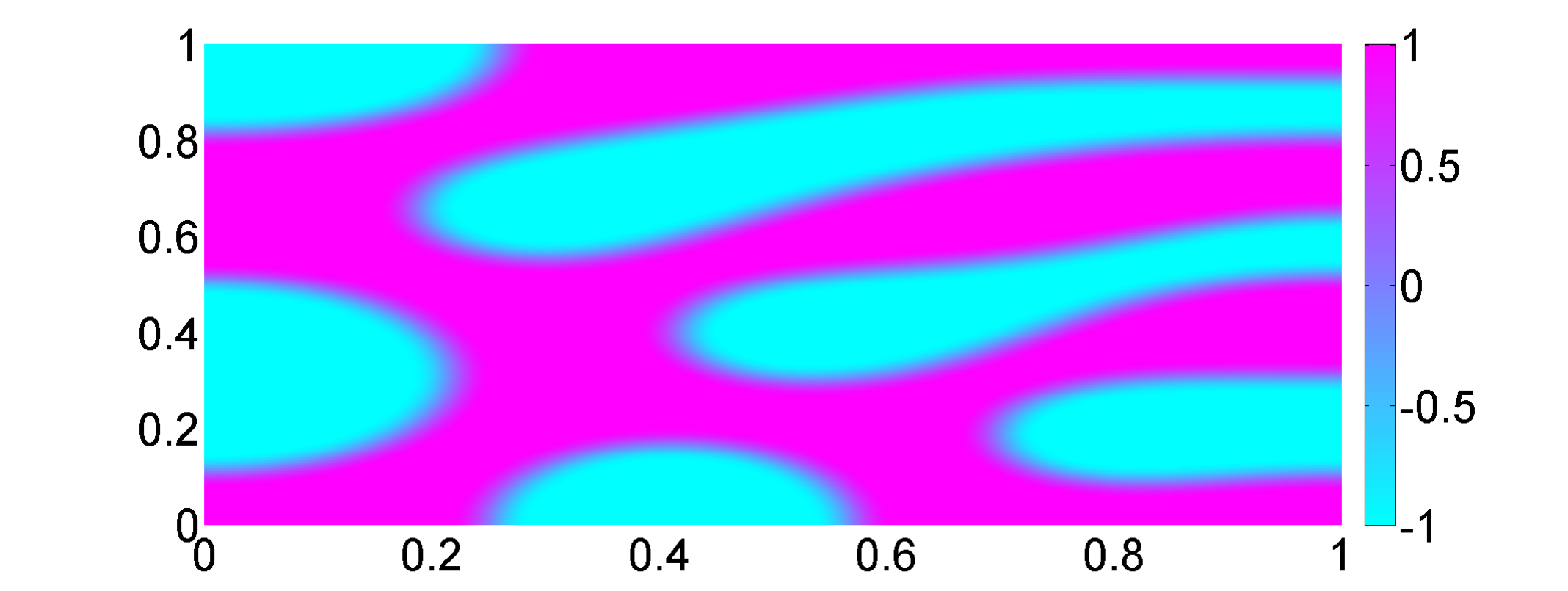}}
\hfill 
\centering 
\subfigure[\label{fig:rand_200}$\tau=200$]{\includegraphics[width=3.4cm, height=3.1cm]{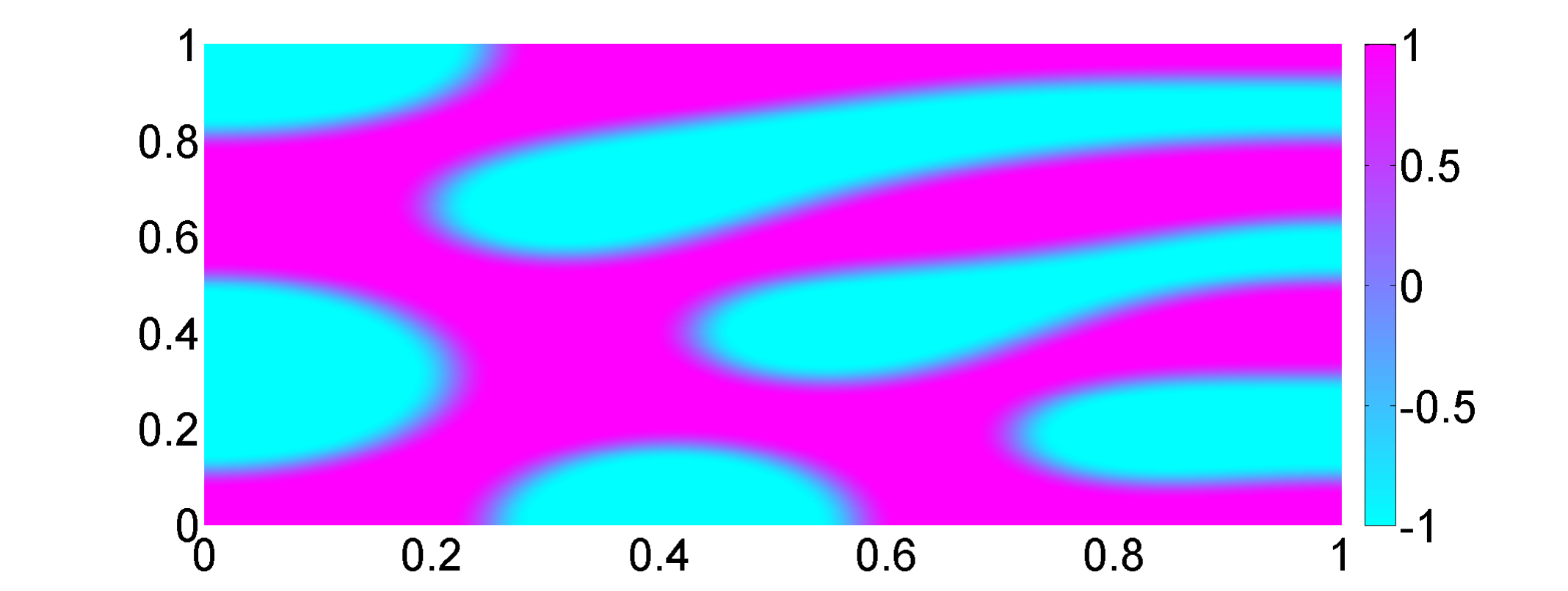}}
\caption{\label{fig:rand} Evolution of Random Initial Active Set configuration}
\end{figure}

\begin{table}
\caption{\label{table:evol_rand}Initial Random Active Set Configuration}
\begin{center}
\begin{tabular}{rrrr|rrr|rrr|rrr}
  &        &         &        & \multicolumn{3}{c|}{\tt bd} &
  \multicolumn{3}{|c}{\tt bdsc} & 
  \multicolumn{3}{|c}{\tt btdsc}                                                                                             \\
  \cline{5-13} 
  $1/h$ & \#tstp & \#trunc & \%trun & it1                     & it2                       & time & it1 & it2 & 
  time  & it1    & it2     & time                                                                                       \\
  \hline
  256    & 1      & 2       & 0.00   &  17                       & 15
  & 18.3     &  19   &  25   & 54.2 & 8 & 10 & 29.8 \\
  & 20     & 13865   & 20.99  &  23                       & 22
  & 29.0     &  24   &  25   & 51.7 & 11 & 11 & 24.2 \\
  & 40     & 25696   & 38.90  &  23                       & 21
  & 25.8     &  20   &  20   & 37.5 & 10 & 10 & 18.7 \\
  & 60     & 31109   & 47.09  &  23                       & 21
  & 25.0     &  19   &  19   & 31.8 & 10 & 10 & 17.6 \\
  & 80     & 34907   & 52.85  &  23                       & 21
  & 24.8     &  19   &  19   & 30.1 & 11 & 10 & 17.9 \\
  & 100    & 37336   & 56.52  &  22                       & 21
  & 24.9     &  19   &  19   & 30.6 & 10 & 9 & 15.7 \\
       & 120    & 39922   & 60.44  &  22                       & 19
       & 21.5     &  17   &  16   & 24.2 & 10 & 9 & 14.5 \\
       & 140    & 40357   & 61.10  &  21                       & 19
       & 21.1     &  17   &  16   & 23.9 & 10 & 9 & 14.8 \\
       & 160    & 40861   & 61.86  &  21                       & 19
       & 20.2     &  17   &  16   & 23.6 & 10 & 9 & 14.6 \\
       & 180    & 41215   & 62.40  &  21                       & 19
       & 20.7     &  17   &  16   & 23.6 & 10 & 9 & 14.6 \\
       & 200    & 41490   & 62.81  &  21                       & 19
       & 20.6     &  17   &  16   & 24.2 & 9 &  9 & 13.9 \\
\hline
       & 1      & 2       & 0.00   &  17                       & 15
       & 46.1     &  19   &  25   & 131.3 & 8 & 10 & 72.3 \\
 400   & 20     & 16136   & 10.03  &  22                       & 22
 & 71.3     &  23   &  26   & 128.4 & 17 & 16 & 56.2 \\
       & 40     & 55886   & 34.75  &  23                       & 23
       & 69.6     &  21   &  21   & 95.7 & 10 & 10 & 47.9 \\
       & 60     & 72514   & 45.09  &  23                       & 21
       & 62.9     &  20   &  20   & 82.2 & 10 & 10 & 42.5 \\
       & 80     & 85496   & 53.16  &  23                       & 21
       & 57.4     &  18   &  18   & 67.7 & 10 & 10 & 42.4 \\
       & 100    & 92787   & 57.70  &  21                       & 21
       & 53.3     &  17   &  16   & 60.0 & 10 & 9 & 35.7 \\
       & 120    & 95995   & 59.69  &  21                       & 19
       & 49.0     &  17   &  16   & 56.7 & 10 & 9 & 35.5 \\
       & 140    & 98593   & 61.31  &  21                       & 19
       & 50.4     &  16   &  16   & 57.0 & 9 & 9 & 33.8 \\
       & 160    & 100733  & 62.64  &  21                       & 19
       & 50.4     &  17   &  16   & 58.4 & 9 & 9 & 33.3 \\
       & 180    & 102625  & 63.82  &  21                       & 19
       & 49.6     &  17   &  20   & 70.9 & 10 & 9 & 34.1 \\
       & 200    & 104522  & 65.00  &  21                       & 19
       & 48.3     &  17   &  16   & 56.2 & 9 & 12 & 47.4 \\
\hline
\end{tabular}
\end{center}
\end{table}

\subsubsection{Evolution of Square}
In this test case, we consider evolution of a square with a diffuse
interface. The initial active set configuration in Figure \ref{fig:square_1}, is
obtained by two squares; the innermost square is prescribed by the lower left
and upper right diagonal ends with coordinates $(0.25, 0.25)$ and $(0.75,
0.75),$ and the outermost square is defined by the coordinates of the diagonal
ends joining $(0.25-10h^2, 0.25-10h^2)$ and $(0.75+10h^2, 0.75+10h^2).$ Thus the
diffuse interface has a thickness of roughly $10h^2.$ In the diffuse interface
region, we consider mixed phases with random values in $[-0.3, 0.5]$ and outside
the diffuse region we presrcibe pure phases of +1 (pink region) and -1 (light
blue region). In Table \ref{tab:evol_square}, we show active set configurations for time steps
$\tau=1,20,40,60,80,100,120,140,160,180,200.$ We observe that for this test
case, the number of truncations remain very high at above 85\%. As for previous
test case, we see significant changes until about $\tau=120,$ after which it
evolves very slowely. In Table \ref{tab:evol_square}, we compare three
preconditioners for various time steps. Here again {\tt btdsc} is the best: it
has least iteration count and small CPU times compared to {\tt bd} and {\tt
  bdsc}. Comparing {\tt bd} and {\tt bdsc}, we find that although {\tt bdsc} has
less iteration count compared to {\tt bd}, {\tt bd} has smaller CPU time. The
reason for this has been explained above. As before, for all three methods, the
number of iterations remain almost constant for various time steps with time
step $\tau \geq 20$. For $h=1/256,$ {\tt bd} is slightly faster compared to {\tt
  bdsc}, and for $h=1/400,$ {\tt bd} is significantly faster compared to {\tt
  bdsc}.
  
\begin{figure}[tbp]
\centering  
\subfigure[\label{fig:square_1}$\tau=1$]{\includegraphics[width=3.4cm, height=3.1cm]{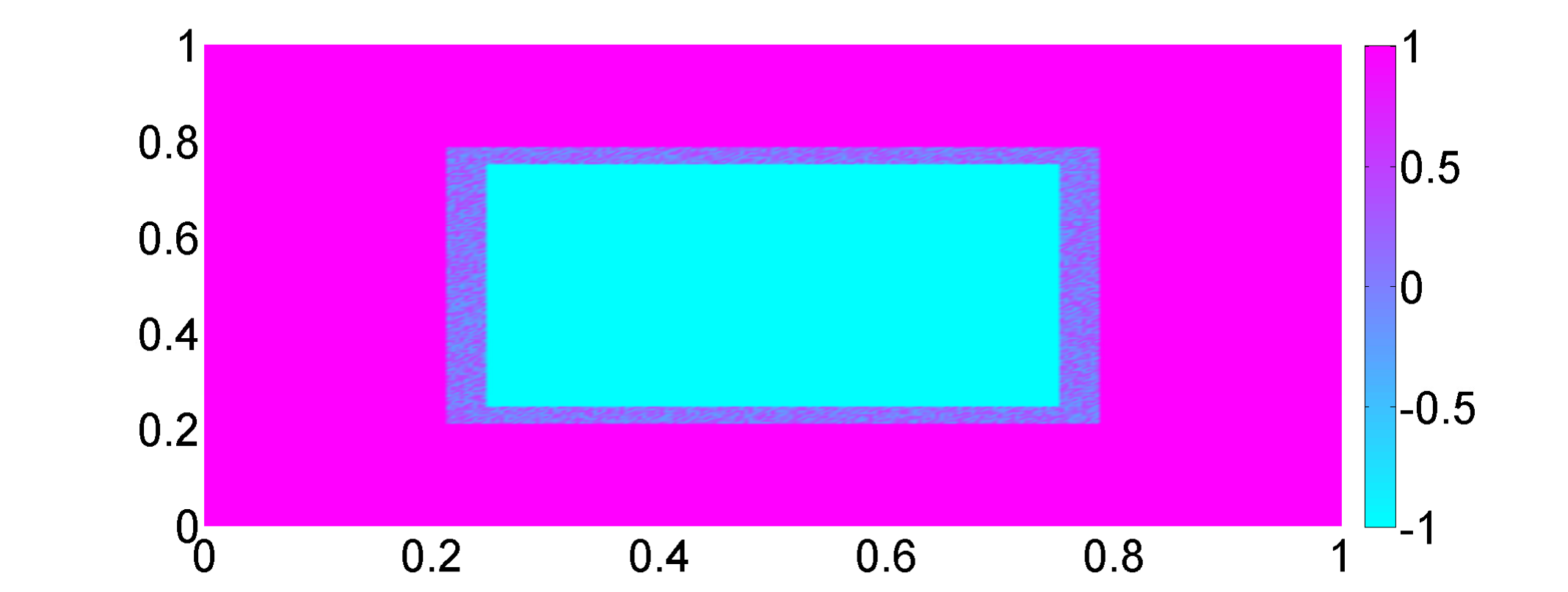}}
\hfill   
\centering  
\subfigure[\label{fig:square_20}$\tau=20$]{\includegraphics[width=3.4cm, height=3.1cm]{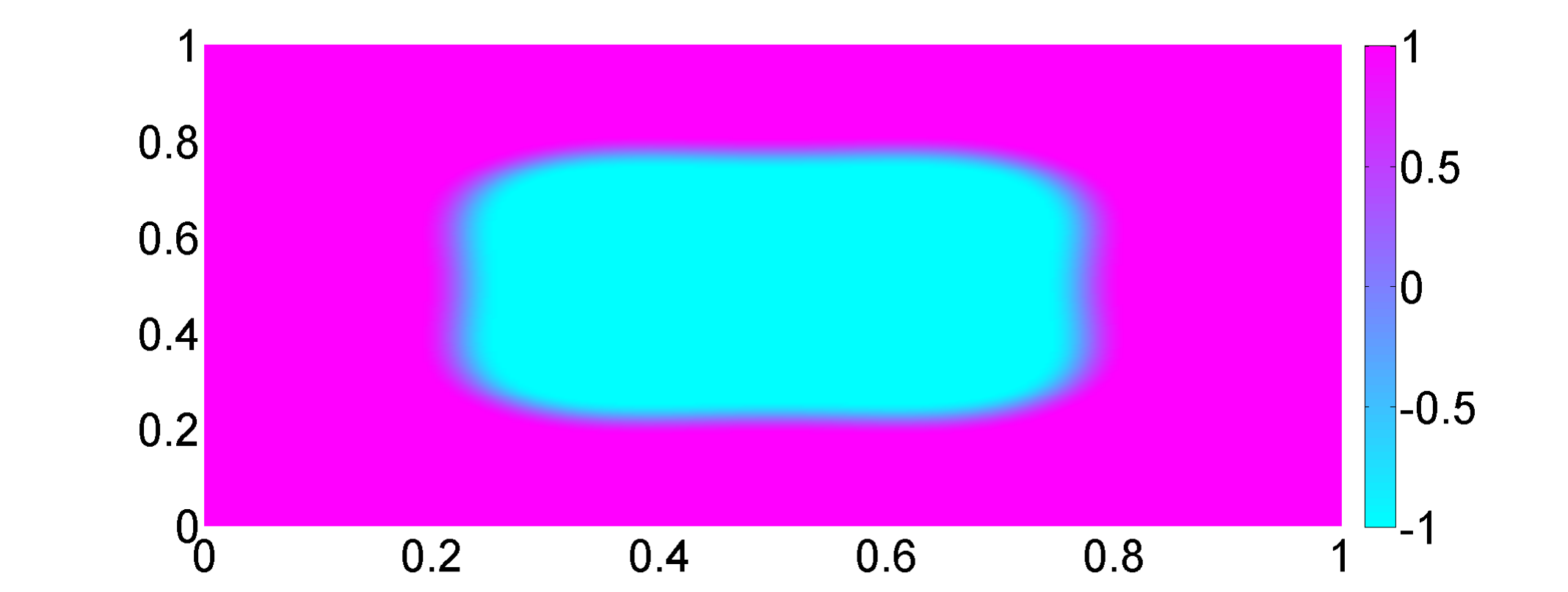}}
\hfill 
\centering  
\subfigure[\label{fig:square_40}$\tau=40$]{\includegraphics[width=3.4cm, height=3.1cm]{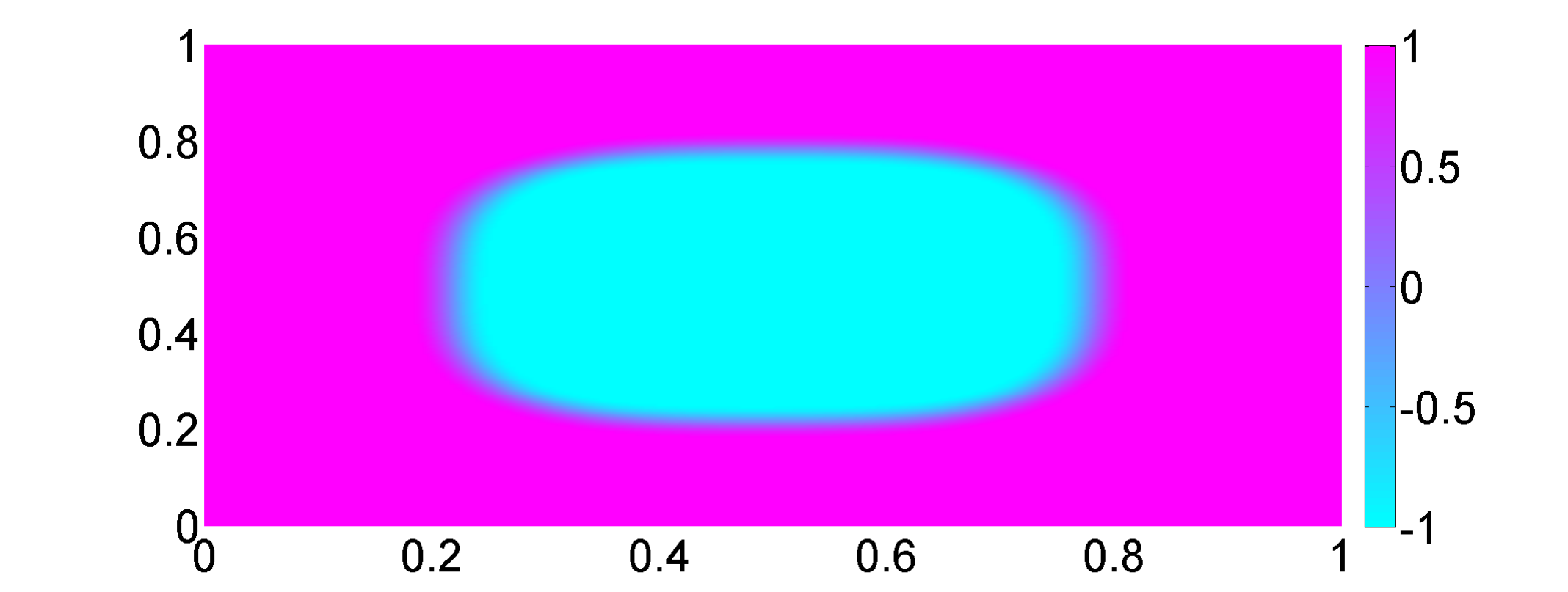}}
\hfill 
\centering  
\subfigure[\label{fig:square_60}$\tau=60$]{\includegraphics[width=3.4cm, height=3.1cm]{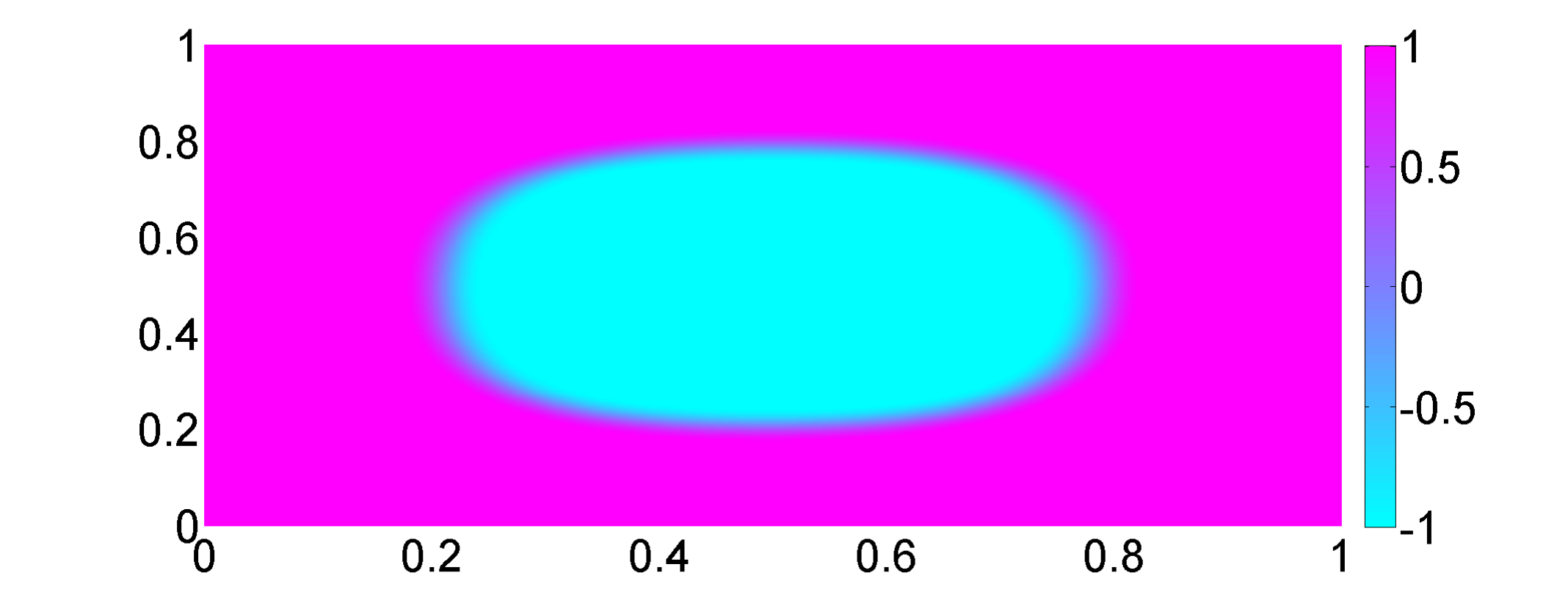}}
\hfill 
\centering   
\subfigure[\label{fig:square_80}$\tau=80$]{\includegraphics[width=3.4cm, height=3.1cm]{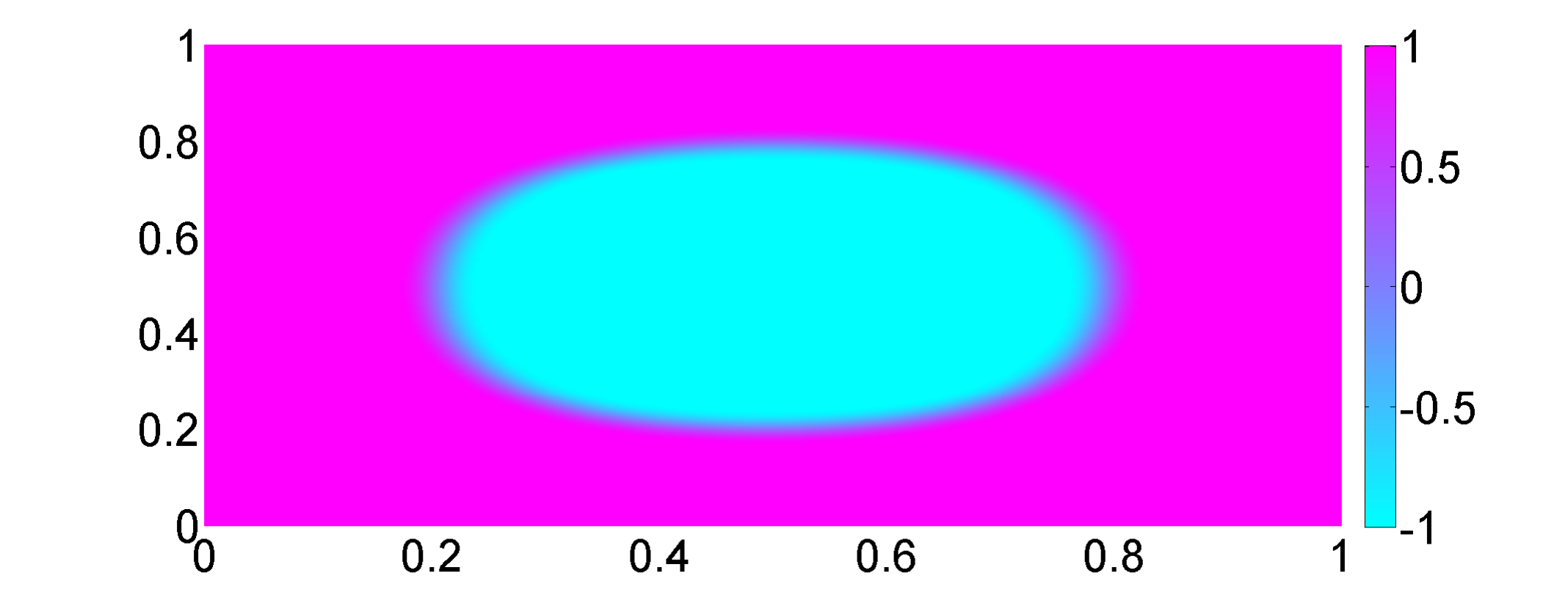}}
\hfill 
\centering   
\subfigure[\label{fig:square_100}$\tau=100$]{\includegraphics[width=3.4cm, height=3.1cm]{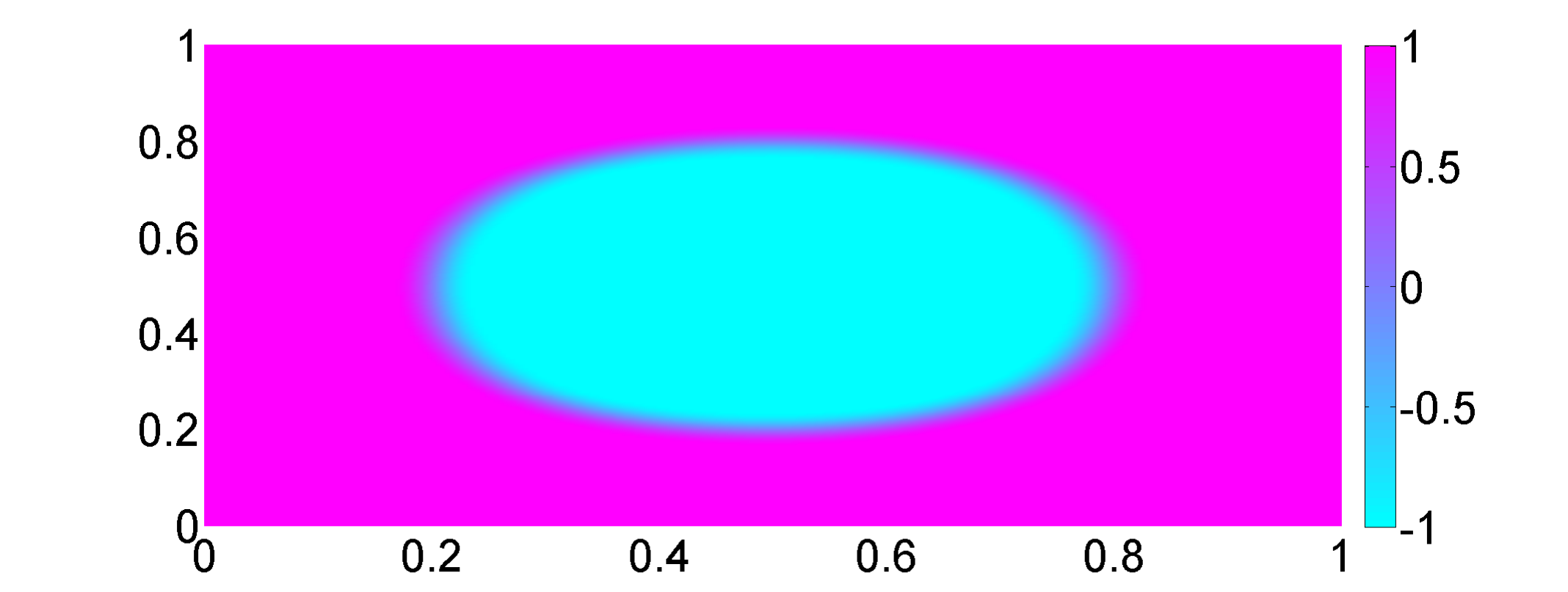}}
\hfill 
\centering   
\subfigure[\label{fig:square_120}$\tau=120$]{\includegraphics[width=3.4cm, height=3.1cm]{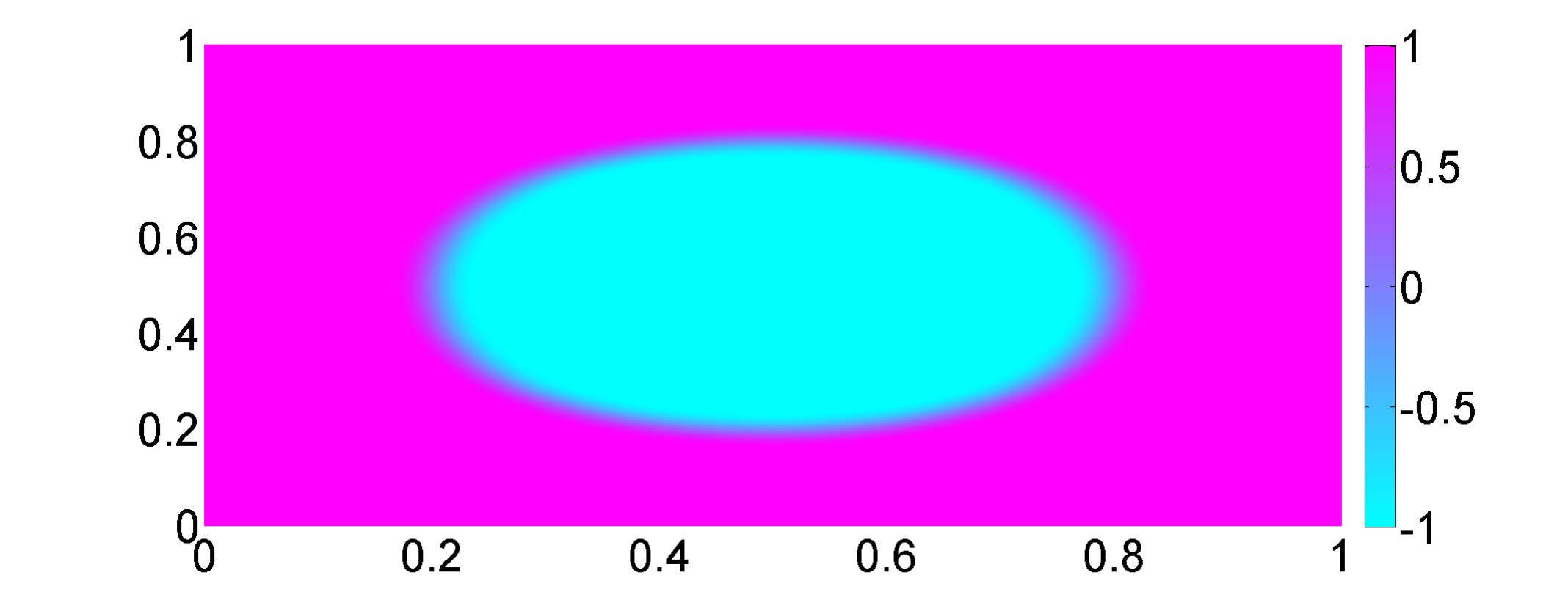}}
\hfill 
\centering   
\subfigure[\label{fig:square_160}$\tau=160$]{\includegraphics[width=3.4cm, height=3.1cm]{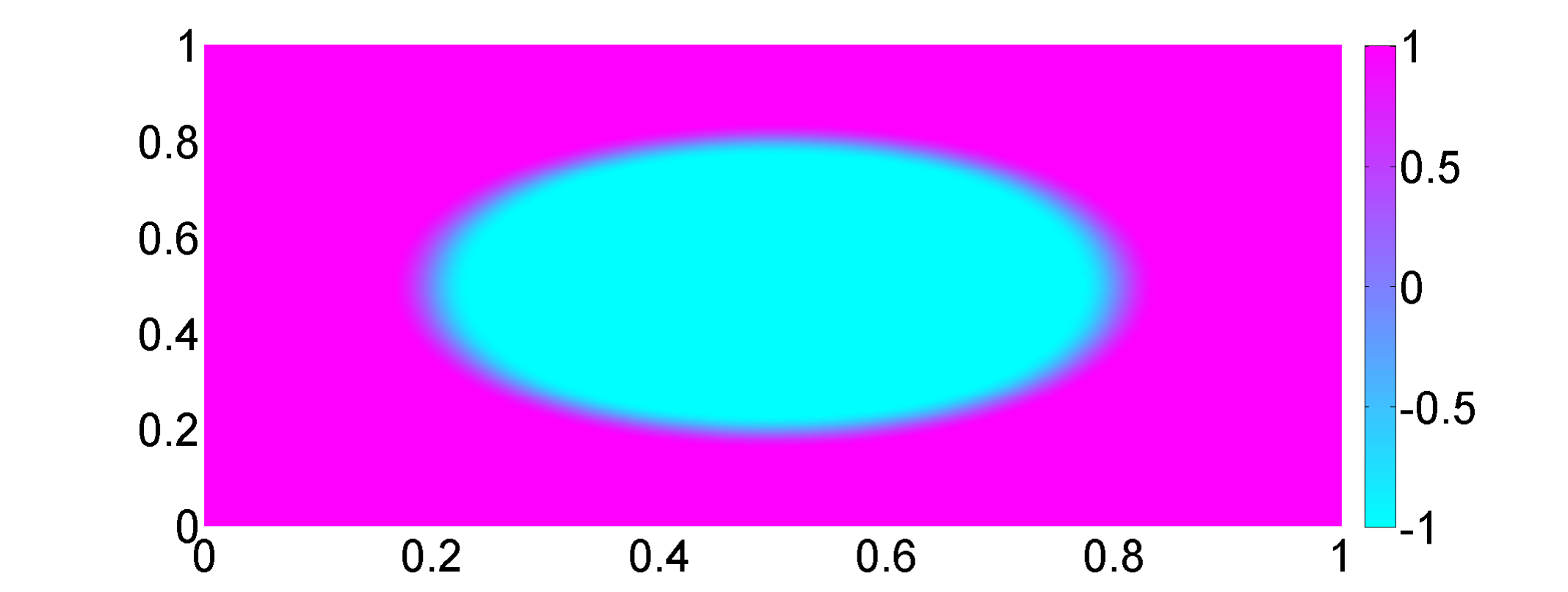}}
\hfill 
\centering     
\subfigure[\label{fig:square_180}$\tau=180$]{\includegraphics[width=3.4cm, height=3.1cm]{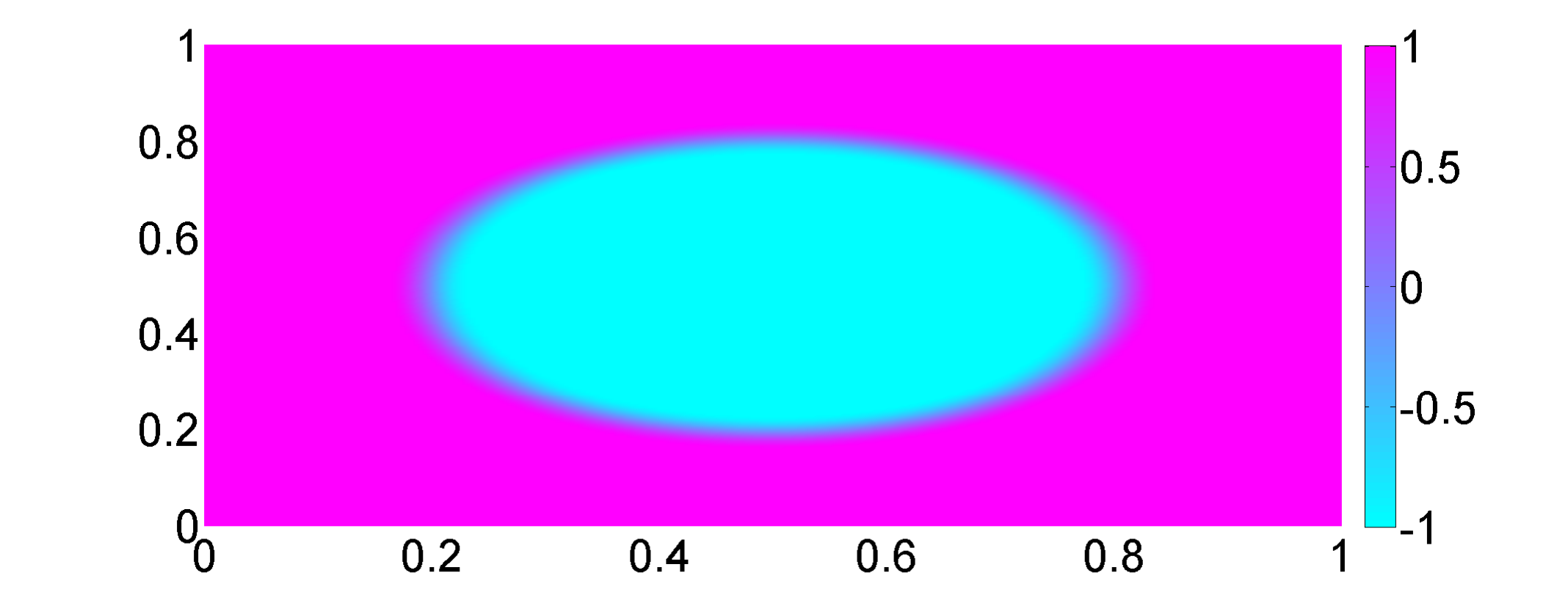}}
\hfill 
\centering     
\subfigure[\label{fig:square_200}$\tau=200$]{\includegraphics[width=3.4cm, height=3.1cm]{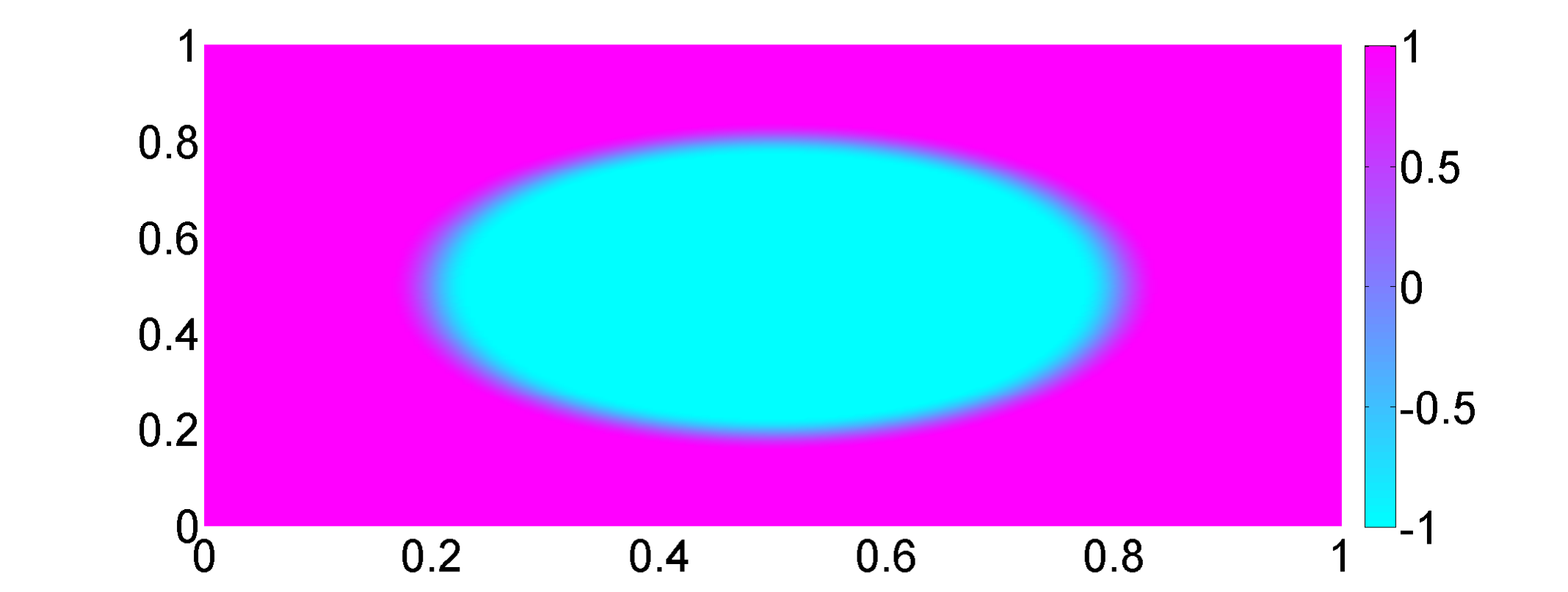}}
\caption{\label{fig:square} Evolution of Initial Square Active Set configuration}
\end{figure}
    
\begin{table}[ht]
\caption{\label{tab:evol_square}Initial Square Active Set Configuration}
\begin{center}
\begin{tabular}{rrrr|rrr|rrr|rrr}
                                      &        &         &        & \multicolumn{3}{|c|}{\tt bd} & 
       \multicolumn{3}{|c|}{\tt bdsc} & 
\multicolumn{3}{|c}{\tt btdsc}                                                                                                                      \\
\cline{5-13} 
 $1/h$                                & \#tstp & \#trunc & \%trun & it1                          & it2 & time & it1 & it2 & time & it1 & it2 & time \\
\hline
256                                   & 1      & 0       & 0.00   & 25                           & 18  & 18.1 & 11  & 14  & 15.8 & 12  & 12  & 16.3 \\
                                      & 20     & 57176   & 86.56  & 21                           & 17  & 17.6 & 16  & 14  & 18.3 & 9   & 8   & 10.8 \\
                                      & 40     & 57358   & 86.84  & 21                           & 17  & 18.3 & 16  & 13  & 18.6 & 9   & 8   & 10.8 \\
                                      & 60     & 57368   & 86.85  & 21                           & 16  & 17.8 & 16  & 13  & 17.7 & 9   & 8   & 10.9 \\
                                      & 80     & 57447   & 86.97  & 21                           & 17  & 17.3 & 16  & 13  & 18.7 & 9   & 8   & 10.8 \\
                                      & 100    & 57426   & 86.94  & 21                           & 16  & 17.3 & 16  & 13  & 17.6 & 9   & 8   & 11.1 \\
                                      & 120    & 57362   & 86.84  & 21                           & 16  & 16.8 & 16  & 13  & 17.2 & 9   & 8   & 10.8 \\
                                      & 140    & 57346   & 86.82  & 21                           & 16  & 17.0 & 16  & 13  & 17.3 & 9   & 8   & 10.6 \\
                                      & 160    & 57366   & 86.85  & 21                           & 16  & 17.7 & 16  & 13  & 17.7 & 9   & 8   & 10.9 \\
                                      & 180    & 57312   & 86.77  & 21                           & 16  & 16.6 & 16  & 13  & 17.2 & 9   & 8   & 10.7 \\
                                      & 200    & 57313   & 86.77  & 21                           & 16  & 17.8 & 16  & 12  & 16.9 & 9   & 8   & 11.2 \\
\hline
 400                                  & 1      & 0       & 0.00   & 27                           & 16  & 59.0 & 14  & 9   & 52.9 & 15  & 9   & 54.9 \\
                                      & 20     & 139237  & 86.58  & 21                           & 16  & 50.3 & 15  & 16  & 64.9 & 9   & 9   & 35.4 \\
                                      & 40     & 139647  & 86.84  & 21                           & 16  & 49.6 & 15  & 15  & 66.9 & 9   & 7   & 49.4 \\
                                      & 60     & 139839  & 86.96  & 20                           & 16  & 60.9 & 16  & 16  & 67.2 & 9   & 7   & 46.0 \\
                                      & 80     & 139823  & 86.95  & 21                           & 16  & 52.5 & 16  & 16  & 67.4 & 9   & 7   & 45.8 \\
                                      & 100    & 139858  & 86.97  & 21                           & 16  & 50.4 & 16  & 14  & 58.1 & 9   & 12  & 45.9 \\
                                      & 120    & 139788  & 86.93  & 21                           & 17  & 52.5 & 16  & 14  & 60.7 & 9   & 9   & 35.5 \\
                                      & 140    & 139731  & 86.89  & 21                           & 16  & 50.7 & 16  & 16  & 64.5 & 9   & 8   & 32.1 \\
                                      & 160    & 139735  & 86.89  & 21                           & 16  & 52.0 & 16  & 15  & 66.2 & 9   & 12  & 47.3 \\
                                      & 180    & 139739  & 86.90  & 21                           & 16  & 51.2 & 16  & 13  & 56.9 & 9   & 12  & 49.4 \\
                                      & 200    & 139720  & 86.89  & 19                           & 16  & 50.6 & 16  & 13  & 53.9 & 9   & 9   & 34.7 \\
\hline
\end{tabular}
\end{center}
\end{table}

\begin{table}[ht]
  \caption{\label{tab:artificial}Random Artificial Active Set Configuration.}
  \begin{center}
    \begin{tabular}{rrr|rrr|rrr|rrr}
            &           &        & \multicolumn{3}{|c|}{\tt bd} & \multicolumn{3}{|c|}{\tt bdsc} & \multicolumn{3}{|c}{\tt btdsc}                \\
      \cline{4-12}     
      $1/h$ & $\epsilon$       & \%trun & it1                          & it2                            & time  & it1 & it2 & time  & it1 & it2 & time  \\
      \hline
      \multirow{12}*{256} 
            & \multirow{4}*{$10^{-2}$} & 0.00   & 16                           & 16                             & 27.9  & 17  & 20  & 47.4  & 9   & 14  & 31.5  \\
            &           & 19.86  & 26                           & 24                             & 31.4  & 16  & 17  & 32.2  & 16  & 17  & 30.8  \\
            &           & 67.11  & 19                           & 18                             & 17.6  & 14  & 13  & 16.2  & 14  & 13  & 15.3  \\   
            &           & 98.73  & 8                            & 6                              & 6.7   & 5   & 5   & 5.6   & 5   & 5   & 5.6   \\  \cline{3-12} 
            & \multirow{4}*{$10^{-5}$}& 0.00   & 18             & 13                             & 8.6  & 16  & 29  & 32.1  & 7   & 10  & 11.9  \\
            &           & 19.86  & 22                           & 21                             & 13.1  & 23  & 23  & 26.9  & 9  & 10  & 11.4  \\
            &           & 67.11  & 26                           & 24                             & 19.2  & 19  & 18  & 17.5  & 10  & 10  & 9.9  \\
            &           & 98.73  & 25                           & 22                             & 29.6  & 13  & 11  & 16.9  & 10  & 9  & 13.9  \\  \cline{3-12} 
            & \multirow{4}*{$10^{-8}$} & 0.00   & 6                            & 4                              & 1.4   & 3   & 6   & 10.5  & 2   & 3   & 4.0   \\
            &           & 19.86  & 6                            & 4                              & 1.9   & 7   & 5   & 6.3   & 4   & 3   & 4.1   \\
            &           & 67.11  & 6                            & 4                              & 3.4   & 7   & 5   & 5.5   & 4   & 3   & 3.9   \\
            &           & 98.73  & 6                            & 4                              & 6.7   & 7   & 3   & 7.5   & 4   & 2   & 5.6   \\
      \hline 
     \multirow{12}*{400} 
            & \multirow{4}*{$10^{-2}$} & 0.0    & 16            & 16                             & 65.3  & 17  & 16  & 111.2 & 8   & 9   & 71.4  \\
            &           & 19.7   & 24                           & 23                             & 69.9  & 17  & 16  & 72.9  & 16  & 17  & 69.9  \\
            &           & 66.6   & 17                           & 16                             & 36.9  & 12  & 12  & 32.9  & 13  & 11  & 32.6  \\
            &           & 98.7   & 6                            & 6                              & 13.7  & 4   & 6   & 13.1  & 5   & 5   & 12.4  \\  \cline{3-12} 
            &\multirow{4}*{$10^{-5}$} & 0.0    & 18                           & 13                             & 31.2  & 16  & 25  & 93.6 & 7   & 10  & 51.4  \\
            &           & 19.7   & 26                           & 25                             & 52.4 & 22  & 22  & 73.9 & 11  & 11  & 39.7  \\
            &           & 66.6   & 35                           & 34                             & 69.7 & 18  & 18  & 45.1 & 16  & 15  & 39.0 \\
            &           & 98.7   & 35                           & 35                             & 98.5  & 17  & 17  & 51.8  & 17  & 16  & 49.4  \\ \cline{3-12} 
            & \multirow{4}*{$10^{-8}$} & 0.0    & 8                            & 5                              & 3.7   & 3   & 6   & 14.2  & 2   & 4   & 9.5   \\
            &           & 19.7   & 8                            & 6                              & 5.3   & 7   & 7   & 16.3  & 4   & 4   & 10.5  \\
            &           & 66.6   & 7                            & 4                              & 9.5   & 7   & 5   & 13.6  & 4   & 3   & 9.1   \\
            &           & 98.7   & 6                            & 4                              & 17.6  & 7   & 5   & 22.6  & 4   & 3   & 15.3  \\
      \hline
\end{tabular}
\end{center}
\end{table}
 
\subsubsection{Artificial Randomly Truncated System} 
This is a non-evolution example. Here we choose $\epsilon=\tau,$ where we study 
the effectiveness of the solver for various values of $\epsilon.$ We artifically create truncations. In Table \ref{tab:artificial}, we show experiments with this test case, and compare the iterations, and CPU time for iterative solve. We notice that for each mesh sizes, we observe a slight increase in the iteration count from $\epsilon=10^{-2}$ to $\epsilon=10^{-5},$ then decreases again for $\epsilon=10^{-8}.$ The iteration counts for $h=1/400$ are comparable to those for $h=1/256.$ As before, {\tt btdsc} remains the fastest, except in some cases, when there are small truncations when {\tt bd} converges faster. In particular, for $\epsilon=\tau=10^{-8},$ {\tt bd} is fastest in most cases.

\section{Conclusion}
For the solution of large scale linear saddle point problems on truncated domains, we studied and compared three preconditioners. We also derived eigenvalue bounds and condition number estimates for untruncated problem, and related those bounds to the related truncated problem whenever possible. The numerical experiments suggest that these are effective preconditioners for such problems. The work is in progress to extend these solvers to three space dimensions, and to multicomponent phase field models. Note that eigenvalue bounds and condition number estimates are independent of space dimensions and should essentially hold for higher dimensions for appropriate discretizations.
\section{Acknowledgement}
This research was partially carried out at IIIT, Hyderabad and at Einstein Foundation, Berlin. 

\bibliography{mybib}
\bibliographystyle{plain}

\end{document}